\documentclass[a4paper,11pt]{article}

\usepackage{makeidx}

\usepackage{latexsym}
\usepackage{amscd}
\usepackage{amsmath}
\usepackage{amssymb}

\usepackage{amsthm}
\usepackage{float}
\usepackage[all]{xy}
\usepackage{graphicx}

\theoremstyle{plain}
\newtheorem{thm}{Theorem}[section]
\newtheorem{cor}[thm]{Corollary}
\newtheorem{lem}[thm]{Lemma}
\newtheorem{cla}[thm]{Claim}

\newtheorem{prop}[thm]{Proposition}

\newtheorem{assum}[thm]{Assumption}
\theoremstyle{definition}
\newtheorem{defn}[thm]{Definition}
\newtheorem{rem}[thm]{Remark}
\newtheorem{exa}[thm]{Example}
\newtheorem{step}{Step}

\newcommand{\Ker}{\mathop{\mathrm{Ker}}\nolimits}
\newcommand{\Image}{\mathop{\mathrm{Im}}\nolimits}
\newcommand{\Coker}{\mathop{\mathrm{Coker}}\nolimits} 

\newcommand{\id}{\ensuremath{\mathop{\mathrm{id}}}}
\newcommand{\Sym}{\mathrm{Sym}}

\newcommand{\Pic}{\mathop{\mathrm{Pic}}\nolimits}
\newcommand{\Ext}{\mathop{\mathrm{Ext}}\nolimits}
\newcommand{\Hom}{\mathop{\mathrm{Hom}}\nolimits}
\newcommand{\RHom}{\mathop{\bR\mathrm{Hom}}\nolimits}
\newcommand{\scRHom}{\mathop{\bR\mathcal{H}om}\nolimits}
\newcommand{\RGamma}{\mathop{\bR\Gamma}\nolimits}
\newcommand{\Lotimes}{\stackrel{\bL}{\otimes}}
\newcommand{\End}{\mathop{\mathrm{End}}\nolimits}
\newcommand{\Spec}{\operatorname{Spec}}
\newcommand{\GL}{GL}

\newcommand{\scEnd}{\mathop{\mathcal{E}nd}\nolimits}
\newcommand{\Coh}{\operatorname{Coh}}

\newcommand{\module}{\operatorname{mod}}
\newcommand{\Per}[1]{\!{\ }^{#1} \! \mathop{\mathrm{Per}}}

\newcommand{\bC}{\ensuremath{\mathbb{C}}}
\newcommand{\bD}{\ensuremath{\mathbb{D}}}

\newcommand{\bL}{\ensuremath{\mathbb{L}}}

\newcommand{\bP}{\ensuremath{\mathbb{P}}}

\newcommand{\bR}{\ensuremath{\mathbb{R}}}

\newcommand{\bZ}{\ensuremath{\mathbb{Z}}}

\newcommand{\scA}{\ensuremath{\mathcal{A}}}

\newcommand{\scC}{\ensuremath{\mathcal{C}}}
\newcommand{\scD}{\ensuremath{\mathcal{D}}}
\newcommand{\scE}{\ensuremath{\mathcal{E}}}
\newcommand{\scF}{\ensuremath{\mathcal{F}}}

\newcommand{\scH}{\ensuremath{\mathcal{H}}}

\newcommand{\scK}{\ensuremath{\mathcal{K}}}
\newcommand{\scL}{\ensuremath{\mathcal{L}}}
\newcommand{\scM}{\ensuremath{\mathcal{M}}}
\newcommand{\scN}{\ensuremath{\mathcal{N}}}
\newcommand{\scO}{\ensuremath{\mathcal{O}}}

\newcommand{\scQ}{\ensuremath{\mathcal{Q}}}

\newcommand{\scU}{\ensuremath{\mathcal{U}}}

\newcommand{\NI}{\noindent}

  \newcommand{\Span}[1]{\left<#1\right>}

\title{Tilting generators via ample line bundles}
\author{Yukinobu Toda and Hokuto Uehara}
\date{}
\begin{document}

\maketitle

\begin{abstract}
It is known that 
a tilting generator on an algebraic variety $X$ gives a derived equivalence 
between $X$ and a certain non-commutative algebra. 
In this paper, we present a method to construct 
a tilting generator from an ample line bundle, 
and construct it in several examples. 
\end{abstract}

\section{Introduction}
Let $D^b(X)$ be the bounded derived category 
of coherent sheaves on an algebraic variety $X$. 
Modern algebraic geometers have
often observed that $D^b(X)$ appears in 
a symmetry connecting two mathematical objects.
For example,
Beilinson~\cite{Bei} finds an example of such phenomena:
he discovers that the derived category 
$D^b(\mathbb{P}^n)$ on the projective space $\bP^n$ 
is equivalent to the derived category 
$D^b(\module \End_{\mathbb{P}^n}(\scE))$
of the abelian category of finitely generated 
right $\End_{\mathbb{P}^n}(\scE)$-modules, where $\scE$ is
the vector bundle
$$
\mathcal{O}_{\mathbb{P}^n}\oplus 
\mathcal{O}_{\mathbb{P}^n}(-1) \oplus \cdots  
\oplus \mathcal{O}_{\mathbb{P}^n}(-n).
$$
We also have the so-called \textit{McKay correspondence} (\cite{BKR}, 
\cite{KaVa}), 
which is a symmetry between complex algebraic geometry and 
representation theory. We now understand the McKay correspondence
 as a derived equivalence
 between an algebraic variety and a non-commutative algebra. 
 
 Van den Bergh proposes a
 generalization of Beilinson's theorem and the McKay correspondence 
through derived Morita theory \cite{Rickard}. 
 \begin{thm}\emph{\bf{\cite[Theorem A]{MVB}}}\label{Intro:1}
 Let $f\colon X\to Y=\Spec R$ be a projective morphism 
 between Noetherian schemes. Assume that $f$ has 
 at most one-dimensional fibers 
 and $\mathbb{R}f_{\ast}\mathcal{O}_X =\mathcal{O}_Y$. 
 Then there is a vector bundle $\mathcal{E}$ on $X$
 such that the functor 
 $$\mathbb{R}\Hom_X(\mathcal{E}, -)\colon 
 D^b(X) \to D^b(\module \End_X(\mathcal{E})),$$
 defines an equivalence of derived categories. 
 \end{thm} 
\NI
Such a vector bundle $\mathcal{E}$ is called
 a \textit{tilting generator}.
In the proof, Van den Bergh uses a globally generated 
ample line bundle $\scL$ on $X$
and constructs $\scE$ from $\scO_X$ and $\scL^{-1}$. 

 Recently, Kaledin~\cite{Kale} proved
 the existence of a tilting generator \'etale locally on $Y$
  when $f\colon X\to Y$ is a crepant resolution and $Y$ has 
 symplectic singularities. 
He uses quite sophisticated tools 
 such as mod $p$ reductions and deformation quantizations, but
 it seems difficult to apply his method 
when $Y$ does not have symplectic singularities.

 The aim of this paper is to generalize Van den Bergh's arguments
 using ample line bundles, and to
 construct a tilting generator in a more general setting.
 In particular, we relax the 
fiber dimensionality assumption.
 One of our main results is:
 \begin{thm}\emph{\bf{[Theorem~\ref{thm:rel.dim2}]}}
 \label{thm:main0}
 Let $f\colon X\to Y=\Spec R$ be a projective morphism between Noetherian 
 schemes and $R$ be a ring of finite type over a field, or a Noetherian complete local ring. Assume that $f$ has at most two-dimensional 
 fibers and $\bR f_{\ast}\scO_X =\scO_Y$. 
 Further assume that 
 there is an ample globally generated 
 line bundle $\mathcal{L}$ on $X$ that satisfies 
 $\mathbb{R}^2 f_{\ast}\mathcal{L}^{-1}=0$. 
 Then there is a tilting vector bundle generating the derived category $D^-(X)$. 
 \end{thm} 

 Our method can apply to more general situations: for instance,
 we can show that there is a tilting generator 
 on $X=T^*G(2,4)$, where $G(2,4)$ is the Grassmann manifold. 
 The variety $X$ admits
  the Springer resolution $f\colon X\to \Spec R$, which 
 has a 4-dimensional fiber.

 The paper is organized as follows.
 In \S \ref{section:ample},
 we show some easy results on ample line bundles, which we use later.
 In \S \ref{section:tilting}, we define tilting generators and
 explain their properties.
 In \S \ref{section: construction}, we present our main construction of 
 tilting generators and the assumptions behind it.
 In \S \ref{section:heart}, we study the heart of a t-structure given in 
 \S \ref{section: construction}. The results in \S \ref{section:heart} are 
 not used in any other sections. 
 In \S \ref{section:2dimensional}, we prove Theorem \ref{thm:main0}
 and find several examples where we can apply Theorem \ref{thm:main0}.
 In \S \ref{section:G(2,4)}, we find a tilting generator 
 of the derived category of the cotangent bundle of 
 the Grassmann manifold $G(2,4)$. 
 In \S \ref{section:auxiliary}, we show an auxiliary result which is
 needed in \S \ref{subsection:gluing}. To prove the result 
 in \S \ref{section:auxiliary},
we require the dualizing complex $D_R$ on $Y$ 
in Theorem \ref{thm:main0}. This requirement is why 
we assume that $Y$ is a scheme of finite type over a field or 
a spectrum of a Noetherian complete local ring.
In the appendix, we apply our result to prove the existence of 
non-commutative crepant resolutions in the sense of 
 Van den Bergh (\cite{nonc}).


\paragraph{Notation and Conventions.}
For a right (respectively, left) Noetherian (possibly non-commutative) 
ring $A$,
$\module A$ (respectively, $A\module$) is the abelian 
category of finitely generated right (respectively, left) 
$A$-modules 
and we set $D^b(A)=D^b(\module A)$,
$D^-(A)=D^-(\module A)$ etc. 
We denote by $A^\circ$ the opposite ring of a ring $A$.

For a Noetherian scheme $X$,
we denote by $D(X)$ (respectively, $D^b(X)$, $D^-(X)$, $\ldots$.) 
the unbounded (respectively, bounded, bounded above, $\ldots$.)
derived category of coherent sheaves.
If $\scA$ is a sheaf of $\scO_X$-algebras,
then we denote by $\Coh \scA$ the category of right coherent 
$\scA$-modules. Put $D^-(\scA)=D^-(\Coh \scA)$.

We also denote by $D_X$ the dualizing complex (if it exists) and
by $\bD_X$ the dualizing functor 
$$
\scRHom _X(-,D_X)\colon D^-(X)\to D^+(X).
$$

For a complex $\scK$ of coherent sheaves on $X$,
we denote by $\tau_{\le p}\scK(=\tau_{<p+1}\scK)$ and 
$\tau_{> p}\scK(=\tau_{\ge p+1}\scK)$ 
the following complexes:
$$
 (\tau_{\le p}\scK)^n= \begin{cases}
                        \scK^n & n < p \\
            \Ker d^p & n = p \\
                     0 & n > p  
\end{cases}
$$
$$
 (\tau_{> p}\scK)^n= \begin{cases}
                    0 & n < p      \\
              \Image d^p & n = p \\
                     
                     \scK^n & n > p .
\end{cases}
$$
Here, $d^{p}\colon \mathcal{K}^{p}\to \mathcal{K}^{p+1}$ is 
the differential.
Similarly we denote by $\sigma_{\le p}\scK(=\sigma_{< p+1}\scK)$ 
and $\sigma_{> p}\scK(=\sigma_{\ge p+1}\scK)$ 
the following complexes:
 $$
 (\sigma_{\le p}\scK)^n= \begin{cases}
                        \scK^n & n \le p \\
                        0 & n>p
\end{cases}
$$
and
$$
 (\sigma_{> p}\scK)^n= \begin{cases}
                         0 & n\le p \\
                          \scK^n & n > p.
\end{cases}
$$
Then there are distinguished triangles in $D(X)$:
$$
\tau_{\le p}\scK \to \scK \to \tau_{>p}\scK\to \tau_{\le p}\scK[1]
$$
and 
$$
\sigma_{> p}\scK \to \scK \to \sigma_{\le p}\scK\to \sigma_{> p}\scK[1].
$$

We denote by $D(X)^{\le p}$ the full subcategory of $D(X)$:
$$
D(X)^{\le p} = \bigl\{ 
\scK \in D(X) \bigm|
\scH^i(\scK)=0 \mbox{ for all } i>p  \bigr\}.
$$
We also define $D(A)^{\ge p}\ldots$ similarly.
  

\paragraph{Acknowledgement.}
Y.T. is supported by J.S.P.S for Young Scientists (No.198007).
H.U. is supported by the Grants-in-Aid 
for Scientific Research (No.17740012). 
H.U. thanks Hiraku Nakajima for useful discussions.


\section{Results on ample line bundles}\label{section:ample}
In this section, we present some easy results on ample line bundles.
Let $f \colon X\to Y=\Spec R$ be a projective morphism 
from a Noetherian scheme to a Noetherian affine scheme.
Suppose that $\mathbb{R}^i f_{\ast}\mathcal{O}_X=0$
for $i>0$ and the fibers of $f$ are
at most $n$-dimensional ($n\ge 0$). 
Assume further that there is an ample, 
globally generated line bundle $\scL $ on $X$, satisfying
\begin{equation}\label{eqn:ample4}
\bR^if_*\scL^{-j}=0
\end{equation}
for $i\ge 2, 0<j<n$.

Take general elements $H_k\in |\scL|$, $1\le k\le n$, and
put $H^k=H_1\cap\cdots\cap H_k, H^0=X$ and $H=H^1$.
Below we often use the following exact sequence:
\begin{align}\label{eqn:Lj}
0\to\scL^{l-1}|_{H^k}\to\scL^{l}|_{H^k}\to\scL^{l}|_{H^{k+1}}\to 0.
\end{align}

\begin{lem}\label{lem:vanishing}
In the above situation,
we have
$$\bR^if_*\scL^j=0$$
for all $i>0, j\ge 0$.
\end{lem}
\begin{proof}
We show the assertion by induction on $n$, the upper bound 
of the dimension of the fibers of $f$.
 The statement obviously holds when $f$ is quasi-finite, 
that is, $n=0$. Next, suppose that $n>0$ and the 
statement holds 
for $n-1$.
 
By (\ref{eqn:ample4}) and (\ref{eqn:Lj}),
we see $\bR^1 f_{\ast}(\mathcal{O}_H)=0$ and
$$
\bR^if_*(\scL^{-j}|_H)=0 
$$
for $i\ge 2, 0\le j<n-1$. 
Hence we can use the induction hypothesis, and conclude
$$
\bR^if_*(\scL^j|_H)=0
$$
for all $i>0, j\ge 0$. Therefore,
there is a surjection
$
\bR^if_*\scL^{j-1}\twoheadrightarrow \bR^if_*\scL^j.
$
Since $\bR^if_*\scO_X= 0$ for $i>0$, 
we obtain the assertion. 

\end{proof}
\noindent
In the application below,
$X$ is always a smooth variety and $-K_X$ is $f$-nef and $f$-big. 
If, furthermore, $X$ is defined over $\bC$, then Lemma 
\ref{lem:vanishing} is 
automatically true by the vanishing theorem.
(cf.~\cite[Theorem~1-2-5]{KMM}.)
Next we see the following:

\begin{lem}\label{lem:properties}
In the above situation, we have
$$
\RHom_X (\scL^{-n},C)\in R\module
$$
for $C\in\Coh X$ with 
$\RHom_X (\bigoplus_{i=0}^{n-1}\scL^{-i},C)=0$,
and
$$
\RHom_X (\scL^{n},C)\in R\module [-n]
$$
for $C\in\Coh X$ with
$\RHom_X (\bigoplus_{i=0}^{n-1}\scL^{i},C)=0$.
\end{lem}

\begin{proof}
Take  $C\in\Coh X$ such that 
$$\RGamma (X,\bigoplus_{i=0}^{n-1}\scL^{i}\otimes C)
\cong \RHom_X (\bigoplus_{i=0}^{n-1}\scL^{-i},C)=0.
$$  
Then we can show from (\ref{eqn:Lj}) that
$
\RGamma (H^k,\bigoplus_{i=k}^{n-1}\scL^{i}\otimes C|_{H^k})=0
$
for $k=0,\ldots,n-1$ inductively. Therefore we have  
$$
\RGamma (X,\scL^{n}\otimes C)\cong \RGamma (H,\scL ^{n}\otimes C|_H)
\cong \cdots \cong \RGamma ({H^n},\scL ^{n}\otimes C|_{H^n}).
$$
Because $H^n$ is relative $0$-dimensional, we obtain
$
\RHom _X(\scL^{-n},C)\in R \module
$ as required.

Take  $C\in\Coh X$ such that 
\begin{align*}
\RGamma (X,\bigoplus_{i=-n+1}^{0}\scL^{i}\otimes C)
&\cong \RHom_X(\bigoplus_{i=0}^{n-1}\scL^{i},C) =0.
\end{align*}
Then we can show from (\ref{eqn:Lj}) that
$
\RGamma (H^k,\bigoplus_{i=-n+k+1}^{0}\scL^{i}\otimes C|_{H^k})=0
$
for $k=0,\ldots,n-1$ inductively.   
Therefore we have 
$$
\RGamma (X,\scL^{-n}\otimes C)\cong \RGamma (H,\scL ^{-n+1}\otimes C|_H)[-1]
\cong \cdots \cong \RGamma ({H^n}, C|_{H^n})[-n].
$$
Because $H^n$ is $0$-dimensional, we obtain $\RHom _X(\scL^{n},C)
\in R\module[-n]$
\end{proof}

The following lemma is fundamental in this paper.
\begin{lem}\emph{\bf{\cite[Lemma 3.2.2]{MVB}}}\label{lem322}
Let $f\colon X\to Y$ be a projective morphism between Noetherian schemes
with at most $n$-dimensional fibers. Assume that $Y$ is affine. Let 
$\scL$ be a globally generated ample line bundle on $X$.
Then $\bigoplus_{i=0}^{n} \scL^{i}$ is a generator of $D^-(X)$
(see the definition of generators in Definition \ref{def:tilting}.)
\end{lem}


\section{Tilting generators}\label{section:tilting}
In this section, we define tilting generators on algebraic varieties.

Let $f \colon X\to Y=\Spec R$ be a projective morphism 
from a Noetherian scheme to an affine Noetherian scheme. 

\begin{defn}\label{def:tilting}
Let $\scE$ be a perfect complex on $X$:
that is, locally $\scE$ is quasi-isomorphic to
 a bounded complex of finitely generated free $\scO_X$-modules.
\begin{enumerate}
\item
$\scE$ is said to be tilting
if $\Hom_X^i(\scE,\scE)=0$ for any $i\ne 0$.
\item
$\scE$ is called a generator of $D^-(X)$ if  
the vanishing $\RHom_X(\scE,\scK)=0$ for $\scK\in D^-(X)$ 
implies $\scK=0$.
\end{enumerate}  
\end{defn}
\begin{exa}
The vector bundle 
$
\scE=\bigoplus_{i=0}^n \mathcal{O}_{\mathbb{P}^n}(-i)
$
on $\bP ^n$ is a tilting generator by Lemma \ref{lem322}.
This fact was first observed by Beilinson~\cite{Bei} .
\end{exa}
\noindent
For a tilting vector bundle $\scE$ on $X$, we denote by $A$ 
the endomorphism algebra $\End_X(\scE)$ and 
define functors:
\begin{align*}
\Phi(-)
&=\RHom_X(\scE,-) \colon D^-(X)\longrightarrow D^-(A), \\ 
\Psi(-)
&=-\Lotimes_{A}\scE \colon D^-(A)\longrightarrow D^-(X).
\end{align*}
Note that $\Psi$ is a left adjoint functor of $\Phi$
and $\Phi\circ\Psi \cong \id_{D^-(A)}$. 

The following lemma explains a characteristic property of    
tilting generators.
The statement is well-known,
but for the reader's convenience, 
we supply the proof.
 
\begin{lem}\label{lemma:tilting}
In the above setting, assume furthermore that $\scE$ is a generator
of $D^-(X)$. Then 
$\Phi$ and $\Psi$ define an equivalence of triangulated categories between 
$D^-(X)$ and $D^-(A)$. This equivalence restricts to an equivalence
 between $D^b(X)$ and $D^b(A)$. 
\end{lem}

\begin{proof}

The isomorphism $\Phi\circ\Psi \cong \id_{D^-(A)}$ implies 
that the cone $C$ of the adjunction morphism 
$\Psi\circ \Phi (\scF)\to \scF$ for $\scF\in D^-(X)$ 
is annihilated by $\Phi$. 
Since $\scE$ is a generator of $D^-(X)$, $C$ is zero. 
In particular $\Psi\circ \Phi \cong \id_{D^-(X)}$:
 that is, $\Phi$ and $\Psi$
define an equivalence of triangulated categories between 
$D^-(X)$ and $D^-(A)$. 

We can show that this equivalence restricts 
to an equivalence between $D^b(X)$ and $D^b(A)$.
It is obvious that $\Phi (\scF)\in D^b(A)$ for $\scF\in D^b(X)$, 
so we only need to check that $\Psi(M)\in D^b(X)$
for any $M\in D^b(A)$. To prove this fact, we may assume $M\in\module A$.
For a sufficiently small integer $m$, consider the map 
$$
\phi\colon \tau_{<m}\Psi (M)\to \Psi (M)
$$
induced by the canonical truncation $\tau$,
 and apply $\Phi$ to it;
$$
\Phi(\phi)\colon \Phi(\tau_{<m}\Psi (M))\to \Phi\circ\Psi (M)\cong M.
$$
Then the map $\Phi(\phi)$ is zero by the choice of $m$.
Hence the map $\phi$ is also zero, since 
$\Phi\colon D^-(X)\to D^-(A)$ gives an equivalence. 
This implies $\Psi (M)\in D^b(X)$.
\end{proof}
 

\section{Main construction}\label{section: construction}

In this section, we show how to construct tilting generators from ample line bundles. 
The main result in this section is Theorem 
\ref{thm:main}.

\subsection{Setting}\label{setting}

Let $f \colon X\to Y=\Spec R$ be a projective morphism 
from a Noetherian scheme to an affine scheme 
of finite type over a field, 
or an affine scheme of a Noetherian complete local ring.
Suppose that $\bR f_*\scO_X=\scO_Y$ and fibers of $f$ are
at most $n$-dimensional. 
Assume furthermore that there is an ample, 
globally generated line bundle $\scL $ on $X$, satisfying
\begin{equation}\label{eqn:ample2}
\bR^if_*\scL^{-j}=0
\end{equation}
for $i\ge 2, 0<j<n$.
Then as shown in Lemma \ref{lem:vanishing},
we have
\begin{equation}\label{eqn:ample3}
\bR^if_*\scL^j=0
\end{equation}
for all $i>0, j \ge 0$. Furthermore, we know that 
$\bigoplus _{i=0}^n\scL^{-i}$ is a generator of $D^-(X)$
by Lemma \ref{lem322}.

\begin{rem}
If we assume that (\ref{eqn:ample2}) holds 
for $i\ge 1$ and $0<j\le n$, then 
$\bigoplus _{i=0}^n\scL^{-i}$ is already a 
tilting generator, so there is nothing left to prove. 
\end{rem}

\subsection{Orientation}\label{subsection:orientation}
For illustrative purposes, before explaining our construction,
we sketch a proof of Theorem \ref{Intro:1}.\\

\noindent
{\it Proof of Theorem \ref{Intro:1}.}
First take the extension corresponding to a set of a generators 
of the $R$-module
$H^1(X,\scL^{-1})$;
\begin{equation}\label{eqn:orientation}
0\to \scL^{-1}\to \scN \to \scO_X^{\oplus r} \to 0.
\end{equation}
Then by a direct calculation, 
we can show that $\scE=\scO_X\oplus \scN$ is a tilting object.
We can also see that $\scE$ is a generator of $D^-(X)$
by Lemma \ref{lem322}.
\qed \\

In the following subsections, we construct tilting vector bundles 
$\scE_k$ inductively as follows.
First take $\scE_0=\scO_X$, which is tilting by the assumption
$\bR f_*\scO_X=\scO_Y$ (or (\ref{eqn:ample3})).
Giving a tilting vector bundle $\scE_{k-1}$ with $0<k\le n-1$,
take the extension (\ref{eqn:short}) as (\ref{eqn:orientation})
and 
define a new tilting vector bundle
$\scE_k$ as $\scE_{k-1}\oplus\scN_{k-1}$.
To construct a tilting generator $\scE_n$, we need 
a slightly more careful treatment, as explained in \S \ref{subsection:gluing}.

\subsection{Inductive construction of tilting vector bundles}
\label{inductive}

Under the setting in \S \ref{setting}, we shall construct tilting
 vector bundles $\scE_k$ for $0\le k\le n-1$
 inductively.
 \begin{step}
 Induction hypotheses. 
 \end{step}
Put $\scE_0=\scO_X$ and fix an integer $k$ with $0<k\le n-1$. 
Assume that we have a tilting vector bundle $\scE_{k-1}$ on $X$.
Let us denote the endomorphism algebra $\End_X \scE_{k-1}$ by 
$A_{k-1}$. We also define the following functors:
\begin{align*}
\Phi_{k-1}(-)
&=\RHom_X(\scE_{k-1},-) \colon D(X)\longrightarrow D(A_{k-1}), \\ 
\Psi_{k-1}(-)
&=-\Lotimes_{A_{k-1}}\scE_{k-1} \colon D^-(A_{k-1})
\longrightarrow D^-(X).
\end{align*}
Note that $\Phi_{k-1}$ restricts to the functor
$\Phi_{k-1}\colon D^{-}(X) \to D^{-}(A)$, giving the 
right adjoint functor of $\Psi_{k-1}$. 

As induction hypotheses, we assume the following. 
\begin{itemize}
\item For any $i\ne 0, 1$ and any $l$ with $0< l\le n-1$, we have 
\begin{equation}\label{eqn1}
\Hom^i_X(\scE_{k-1},\scL^{-l})=0. 
\end{equation}
\item For any $i\ne 0$ and any $l$ with $k-1\le l\le n$, we have 
\begin{equation}\label{eqn2}  
\Hom_X ^i(\scL^{-l},\scE_{k-1})=0. 
\end{equation}
\end{itemize}
Note that if $k=1$, (\ref{eqn1}) and (\ref{eqn2}) 
hold by (\ref{eqn:ample2}) and (\ref{eqn:ample3}). 

\begin{step}\label{step:construction}
Construction of $\scE_{k}$. 
\end{step}

Take a free $A_{k-1}$ resolution of $\Phi_{k-1}(\scL^{-k})$
and denote it by $P_{k-1} $.
Since $\scH^{i}(\Phi_{k-1}(\scL^{-k}))=0$ unless $i=0, 1$ by 
(\ref{eqn1}), we can take $P_{k-1}$ satisfying 
$P_{k-1}^i=0$ for $i\ge 2$. 
We obtain a natural morphism
$\sigma_{\ge 1}(P_{k-1} )\to P_{k-1} $,
and hence we have a morphism
$\Psi_{k-1}(\sigma_{\ge 1}(P_{k-1} ))
\to \scL^{-k}$, since $\Psi_{k-1}$ 
is a left adjoint functor of $\Phi_{k-1}$.
Define an object $\scN_{k-1}\in D^-(X)$ 
to be the cone of this morphism; 
\begin{align}\label{eq:cone}
\Psi_{k-1}(\sigma_{\ge 1}(P_{k-1} )) \to \scL^{-k}
 \to \scN_{k-1} \to \Psi_{k-1}
(\sigma_{\ge 1}(P_{k-1} ))[1],
\end{align}
and we also define 
$$\scE_k=\scE_{k-1}\oplus \scN_{k-1}.$$
Applying $\Phi_{k-1}$ to (\ref{eq:cone}) and using the 
isomorphism, 
$$
\Phi_{k-1}\circ\Psi_{k-1}\cong {\id} _{D^-(A_{k-1})},
$$
 we have $\Phi_{k-1}(\scN_{k-1}) 
\cong \sigma_{<1}(P_{k-1} )$.
Furthermore, we know that $\Psi_{k-1}(\sigma_{\ge 1}(P_{k-1} ))$ is 
isomorphic to an object of the form 
$\scE_{k-1}^{\oplus r_{k-1}}[-1]$ for some $r_{k-1}\ge 0$. Hence, 
there is a short exact sequence of coherent sheaves; 
\begin{equation}\label{eqn:short}
0\to \scL^{-k} \to \scN_{k-1} \to \scE_{k-1}^{\oplus r_{k-1}}\to 0.
\end{equation}
Consequently, $\scN_{k-1}$ and $\scE_k$ are vector bundles on $X$.

\begin{step}
$\scE_k$ satisfies the induction hypotheses. 
\end{step}
We shall check below that $\scE_k$ has similar properties to
(\ref{eqn1}) and (\ref{eqn2}).

\begin{cla}\label{cla:tilting1}
$\Hom_X^i(\scE_k,\scL^{-l})=0$ 
for any $i\ne 0,1$ and any $l$ with $0< l\le n-1$
\end{cla}

\begin{proof}
Claim \ref{cla:tilting1} follows from (\ref{eqn:ample2}), 
(\ref{eqn:ample3}), 
(\ref{eqn1})
and the long exact sequence 
$$ \to \Hom_X^i(\scE_{k-1}^{\oplus r},\scL^{-l})\to
\Hom_X^i(\scN_{k-1},\scL^{-l})\to \Hom_X^i(\scL^{-k},\scL^{-l}) \to .
$$
\end{proof}

\begin{cla}\label{cla:tilting2}
$\Hom_X^i(\scL^{-l},\scE_k)=0$ 
for any $i\ne 0$ and any $l$ with $k\le l\le n$.
\end{cla}

\begin{proof}
Claim \ref{cla:tilting2} follows from 
(\ref{eqn:ample3}), (\ref{eqn2})
and the long exact sequence 
$$
\to \Hom_X^i(\scL^{-l},\scL^{-k})\to
\Hom_X^i(\scL^{-l}, \scN_{k-1})\to 
\Hom_X^i(\scL^{-l}, \scE_{k-1}^{\oplus r}) \to.
$$
\end{proof}

\begin{cla}\label{cla:tilting3}
$\scE_k$ is a tilting object.
\end{cla}

\begin{proof}
From $\Phi_{k-1}(\scN_{k-1})\cong 
\sigma_{<1}(P_{k-1} )$, we obtain 
\begin{equation}\label{eqn:tilting1}
\Hom_X^i(\scE_{k-1},\scN_{k-1})=\scH^i(\Phi_{k-1}(\scN_{k-1}))=0
\end{equation}
for all $i\ne 0$.
By (\ref{eqn2}) and the long exact sequence
$$\to \Hom_X^i(\scE_{k-1}^{\oplus r},\scE_{k-1})\to
\Hom_X^i(\scN_{k-1},\scE_{k-1})\to \Hom_X^i(\scL^{-k},\scE_{k-1}) \to,
$$
we have
\begin{equation}\label{eqn:tilting2}
\Hom_X^i(\scN_{k-1},\scE_{k-1})=0
\end{equation}
for all $i\ne 0$. Finally, by Claim \ref{cla:tilting2}, (\ref{eqn:tilting1})
and the long exact sequence
$$
\to \Hom_X^i(\scE_{k-1}^{\oplus r},\scN_{k-1})\to
\Hom_X^i(\scN_{k-1},\scN_{k-1})\to \Hom_X^i(\scL^{-k},\scN_{k-1}) \to,
$$
we have
\begin{equation}\label{eqn:tilting3}
\Hom_X^i(\scN_{k-1},\scN_{k-1})=0
\end{equation}
for all $i\ne 0$.
The equalities (\ref{eqn:tilting1}), (\ref{eqn:tilting2}) and 
(\ref{eqn:tilting3}) imply that 
$\scE_k$ is a tilting object.
\end{proof}
By induction on $k$,
we can construct a tilting vector bundle $\scE_{n-1}$. 

\begin{rem}\label{rem:process}
We cannot apply our method in this subsection
to construct $\scE_n$. 
In Step~\ref{step:construction}, 
we need the vanishing of $\Hom^i_X(\scE_{n-1}, \scL^{-n})$ for $i\ge 2$. 
However this is not guaranteed by the induction hypothesis (\ref{eqn1}).
\end{rem}


\subsection{Gluing t-structures}\label{subsection:gluing}
The vector bundle $\scE_{n-1}$ does not 
generate the category $D^{-}(X)$ yet (see Lemma \ref{lem:generator}),
so we need one more step to construct a tilting generator 
$\scE_{n}$.
As we mentioned in Remark~\ref{rem:process}, 
a similar method in \S \ref{inductive} does not work. 
In this subsection, we make some assumptions
and construct a tilting generator $\scE_{n}$ of $D^-(X)$.  

As in \S \ref{inductive}, we
 define as $A_{n-1}=\End_X \scE_{n-1}$ and 
\begin{align*} 
\Phi_{n-1}(-)&=\RHom_X(\scE_{n-1},-)
\colon D(X)\longrightarrow D(A_{n-1}), \notag\\  
\Psi_{n-1}(-)&=-\Lotimes_{A_{n-1}}\scE_{n-1} \colon D^-(A_{n-1})
\longrightarrow D^-(X).\notag
\end{align*}
Take a free $A_{n-1}$ resolution $P_{n-1}$ of $\Phi_{n-1}(\scL^{-n})$.
Then each
$\scH ^i(P_{n-1} )$ vanishes for $i<0$ but does not necessarily 
vanish for $i\ge 2$. (cf.~Remark~\ref{rem:process}.)
As in Step~\ref{step:construction} in \S \ref{inductive}, 
define an object $\scN_{n-1}\in D^b(X)$ such that $\scN_{n-1}$ fits into a 
triangle
\begin{equation}\label{eqn:triangle}
\Psi_{n-1}(\sigma_{\ge 1}(P_{n-1} )) \to \scL^{-n}
 \to \scN_{n-1} \to \Psi_{n-1}
(\sigma_{\ge 1}(P_{n-1} ))[1].
\end{equation}
Note that $\scN_{n-1}$ is a perfect complex, since so is 
$\Psi_{n-1}(\sigma_{\ge 1}(P_{n-1} ))$.
We again define 
$$\scE_n=\scE_{n-1}\oplus\scN_{n-1}.$$
Although we cannot conclude that $\scE_n$ is tilting, 
we consider the 
functor $\Phi_n(-)=\RHom_X(\scE_{n},-)$.

Let us define $\scC_k$
for $0\le k\le n$ to be the full subcategory of 
the unbounded derived category $D(X)$, 
$$\scC_k =\{ \scK \in D(X) \mid \Phi_{k}(\scK) =0\}.$$
\begin{lem}\label{lem:generator}
Let $k$ be an integer such that $0\le k \le n$. Then
$\scK\in\scC_{k}$ if and only if 
$$
\RHom_X(\bigoplus _{i=0}^k\scL^{-i},\scK)=0.
$$
In particular, $\scE_{n}$ is a generator of $D^-(X)$. 
\end{lem}
\begin{proof}
The proof proceeds by induction on $k$.
First, note that the statement is true for $k=0$, since $\scO_X=\scE_0$.
For $0<k\le n-1$, we obtain from (\ref{eqn:short}) that
\begin{align*}
\scK\in\scC_k &\iff \RHom_X (\scE_{k-1},\scK)
\cong \RHom_X (\scL^{-k},\scK)=0 \\
 &\iff \scK\in\scC_{k-1} \mbox{ and } \RHom_X (\scL^{-k},\scK)=0.
\end{align*}
For $k=n$, 
we have a similar conclusion by (\ref{eqn:triangle}), since
each term of the complex 
$\Psi_{n-1}(\sigma_{\ge 1}(P_{n-1} ))[1]$ 
is a direct sum of $\scE_{n-1}$. 

Suppose that 
$
\RHom_X(\scE_n,\scK)=0
$
for $\scK \in D^-(X)$. Then, the assertion we proved above and  
Lemma \ref{lem322} imply that $\scK =0$, which implies the last statement.
\end{proof}
%
 
 
\begin{rem}\label{rem:easycase}
\begin{enumerate}
\item
In the setting in \S \ref{setting},
assume furthermore 
\begin{align}\label{eq:assume}
\bR^if_*\scL^{-n}=0
\end{align}
for $i\ge 2$: that is,
the vanishing in (\ref{eqn:ample2}) for $j=n$. 
Then we can show that $\scE_{n}$ 
is a tilting vector bundle that generates $D^-(X)$ as follows:
In this case, we can show $\Hom^i_X(\scE_{n-1}, \scL^{-n})=0$ for $i\ge 2$
as Claim \ref{cla:tilting2} and so
the inductive construction in \S \ref{inductive} works for $\scE_n$ 
(see Remark \ref{rem:process}).
By the lemma above, $\scE_n$ is a generator.

In particular, in this extra condition (\ref{eq:assume}) for $n=2$, 
our main Theorem~\ref{thm:main0} becomes rather obvious.
\item
In (i),
there is a short exact sequence of coherent sheaves
\begin{equation}\label{eqn:shortk}
0\to \scL^{-k} \to \scN_{k-1} \to \scE_{k-1}^{\oplus r_{k-1}}\to 0
\end{equation}
for all $k$ with $1\le k\le n$, and some $r_{k-1}\ge 0$. 
Moreover we have
$$
\scE_n=\scO_X\oplus\bigoplus_{k=0}^{n-1}\scN_k.
$$
We can easily see that 
the dual vector bundle $\scE^\vee$ of $\scE$
is also a tilting generator of $D^-(X)$. 
\end{enumerate}
\end{rem}

Let us return to the situation in \S \ref{setting}.
Instead of assuming (\ref{eq:assume}), we 
shall work under the following assumption until the end of 
this section. 

\NI
\begin{assum}\label{assum:induction}
For an object
$\scK \in D(X)$, 
if we have the equality
$$
\RHom_X(\bigoplus_{i=0}^{n-1}\scL^{-i},\scK)=0,
$$ 
then the equality
$$
\RHom_X (\bigoplus_{i=0}^{n-1}\scL^{-i},\scH^k(\scK))=0
$$
holds for all $k$.
\end{assum}

In \S \ref{section:2dimensional} and \S \ref{section:G(2,4)},
 we will study the cases 
where Assumption~\ref{assum:induction} holds. 
Assumption~\ref{assum:induction} means that
$\scK \in\scC_{n-1}$ implies 
$\scH^k(\scK)\in \scC_{n-1}$ for all $k$. 
Then we can define a t-structure on $\scC_{n-1}$ induced by 
the standard one on $D(X)$.
Next we introduce the triangulated category
$$
D^{\dag}(X)= \left\{ \scK \in D(X) \bigm|
\Phi_{n-1}(\scK)\in D^b( A_{n-1}) \right\}.
$$
Note that $\Psi_{n-1}$ defines a functor from 
$D^b(A_{n-1})$ to $D^{\dag}(X)$, 
since $\Phi_{n-1}\circ \Psi_{n-1}\cong \id_{D^b(A_{n-1})}$.
Hence it is a left adjoint functor of 
$\Phi_{n-1}\colon D^{\dag}(X)\to D^b(A_{n-1})$.
The advantage of considering $D^{\dag}(X)$ is the existence 
of a right adjoint functor 
$$\Psi_{n-1}'\colon D^b(A_{n-1}) \to D^{\dag}(X)$$
of $\Phi_{n-1}\colon D^{\dag}(X)\to D^b(A_{n-1})$ 
(see Lemma \ref{lem:adjoint}). 
Therefore, we can construct a new t-structure on 
$D^{\dag}(X)$ by gluing 
the t-structure on $\scC_{n-1}$ (with perversity $p\in \mathbb{Z}$) 
and the standard t-structure on
$D^b(A_{n-1})$ via the exact triple of triangulated categories 
\cite[pp.~286]{GM};
$$
\scC_{n-1}\stackrel{i_{n-1}}{\to} D^{\dag}(X)\to D^b(A_{n-1}).
$$ 
Let $i_{n-1}^{\ast}$, $i_{n-1}^{!}\colon D^{\dag}(X) \to \scC_{n-1}$
 be the left and right 
adjoint functors of the inclusion functor $i_{n-1}$ respectively, 
whose existence follows from 
the existence of the right and left adjoint functors of $\Phi_{n-1}$. 
Specifically, $(i_{n-1}^{\ast}, i_{n-1}^{!})$ are constructed 
so that there are distinguished triangles
\begin{align*}
& \Psi_{n-1}\circ\Phi_{n-1}(E) \to E \to i_{n-1}^{\ast}E, \\
& i_{n-1}^{!}(E) \to E \to \Psi_{n-1}'\circ\Phi_{n-1}(E)
\end{align*}
for any $E\in D^{\dag}(X)$. 
We, therefore, obtain the new t-structure on $D^{\dag}(X)$:
\begin{align*}
^{p}\scD^{\le 0} &=\{ \scK \in D^{\dag}(X) \mid 
\Phi_{n-1}(\scK) \in D^b(A_{n-1})^{\le 0},\ i_{n-1}^{\ast}\scK \in 
\scC_{n-1}^{\le p}\}, \\
^{p}\scD^{\ge 0} &=\{ \scK \in D^{\dag}(X) \mid 
\Phi_{n-1}(\scK) \in D^b(A_{n-1})^{\ge 0},\ i_{n-1}^{!}\scK \in 
\scC_{n-1}^{\ge p}\}.
\end{align*}
Here, $p$ is an integer that determines the \emph{perversity} of the t-structure
and we denote 
$\scC_{n-1}^{\le p}=\scC_{n-1}\cap D(X)^{\le p}$ 
and 
$\scC_{n-1}^{\ge p}=\scC_{n-1}\cap D(X)^{\ge p}$.
 The heart of the above t-structure
is called the category
of \emph{perverse coherent sheaves} (cf. \cite{Br1}):

$$
\Per{p}(X/A_{n-1})= 
\left\{ 
\scK \in D^{\dag}(X) \
\begin{array}{|l}
\Phi_{n-1}(\scK)\in \module A_{n-1} \mbox{ and } \\
i_{n-1}^{\ast}\scK \in \scC_{n-1}^{\le p},\
i_{n-1}^{!}\scK \in \scC_{n-1}^{\ge p}
\end{array} 
\right\}.
$$
\begin{rem}\label{rem:notsoeasy}
Note that 
since the functor 
$\Phi_{n-1}\colon D^b(X)\to D^b(A_{n-1})$ does not necessarily have 
a right adjoint functor, we cannot construct 
the perverse t-structure on $D^b(X)$ in a similar way. 
However we will see in \S \ref{section:heart} that  
$\Per0(X/A_{n-1})$ is in fact the heart of a bounded 
t-structure on $D^b(X)$. 
\end{rem}

\begin{rem}
The condition
$i_{n-1}^{\ast}\scK \in \scC_{n-1}^{\le p}$
(resp.~
$i_{n-1}^{!}\scK \in \scC_{n-1}^{\ge p}$) 
is equivalent to the condition
\begin{align}\label{eq:equiv}
\Hom_X(\scK, C)=0 \quad (\mbox{ resp. }\Hom_X(C, \scK)=0)
\end{align}
for any $C\in \scC_{n-1}^{\ge p+1}$ (resp.~$C\in \scC_{n-1}^{\le p-1}$). 
If $\scK \in D^b(X)$, it is enough to check 
(\ref{eq:equiv}) for $C\in \scC_{n-1}\cap \Coh X[j]$
with $j<-p$ (resp.~$j>-p$).
\end{rem}

\begin{cla}\label{cla:N_n0}
The object 
$\scN_{n-1}$ belongs to $\Per0 (X/A_{n-1})$.
\end{cla}
\begin{proof}
Since $\scN_{n-1}\in D^b(X)$, it is enough to check 
the following;
\begin{align}
\label{check1}& \Phi_{n-1}(\scN_{n-1}) \in \module A_{n-1}, \\
& \Hom_X^i(\scN_{n-1}, C)=0 \mbox{ for }i<0 \mbox{ and }
\label{check2}C\in \scC_{n-1}\cap \Coh X,\\
\label{check3}& \Hom_X^i(C, \scN_{n-1})=0 \mbox{ for }i<0 \mbox{ and }
C\in \scC_{n-1}\cap \Coh X. 
\end{align}
First let us check (\ref{check1}). By the triangle (\ref{eqn:triangle}), 
we have $\Phi_{n-1}(\scN_{n-1})\cong \sigma_{\le 0}P_{n-1}$,
hence 
$\scH^i(\Phi_{n-1}(\scN_{n-1}))=0$ for $i>0$. For $i<0$, we have 
\begin{align*}
\scH^{i}(\Phi_{n-1}(\scN_{n-1}))  
&\cong \scH^{i}(\Phi_{n-1}(\scL^{-n})) \\
&=0, 
\end{align*}
since $\scE_{n-1}$ and $\scL^{-n}$ are vector bundles on $X$. 
Therefore (\ref{check1}) holds.  
Next for $C\in \scC_{n-1} \cap \Coh X$,  
we have 
\begin{align}\label{eq:check2}
\Hom_X^i(\scN_{n-1}, C) \cong \Hom_X^i(\scL^{-n}, C)
\end{align}
for any $i$ by the triangle (\ref{eqn:triangle}).
 Therefore (\ref{check2}) follows.
Finally we check (\ref{check3}). Since $\scL^{-n}$ and 
$\Psi_{n-1}(\sigma_{\ge 1}(P_{n-1} )[1])$ belong to 
$D(X)^{\ge 0}$, we have 
\begin{align*}
\Hom_X ^i(C, \scL^{-n})
&\cong \Hom_X ^i(C, \Psi_{n-1}(\sigma_{\ge 1}(P_{n-1} )[1])) \\
&=0
\end{align*}
for $i<0$ and $C\in\scC_{n-1}\cap\Coh X$.
By the triangle (\ref{eqn:triangle}), 
(\ref{check3}) also follows. 
\end{proof}

\begin{cla}\label{cla:N_n}
For $i>0$ and $B\in\Per0 (X/A_{n-1})$, we have 
\begin{equation}\label{eqn:projective}
\Hom_X ^i(\scN_{n-1},B)=0.
\end{equation} 
In particular, $\scN_{n-1}$ is a projective object of $\Per0 (X/A_{n-1})$. 
\end{cla}

\begin{proof}
We have a triangle
\begin{align}\label{tri2}
\Psi_{n-1}\circ\Phi_{n-1}(B)\to B\to i_{n-1}^{\ast}B.
\end{align}
By the definition of $\Per0(X/A_{n-1})$, we have 
$i_{n-1}^{\ast}B\in\scC_{n-1}^{\le 0}$. 
To see (\ref{eqn:projective}),
it suffices to show
%
%
\begin{equation}\label{eqn:projective2}
\Hom_X ^i(\scN_{n-1}, i_{n-1}^{\ast}B)=0
\end{equation}
and
\begin{equation}\label{eqn:projective2'}
\Hom_X ^i(\scN_{n-1}, \Psi_{n-1}\circ\Phi_{n-1}(B))=0
\end{equation}
for $i>0$. To prove (\ref{eqn:projective2}),
it is enough to show 
$$
\Hom_X^i(\scN_{n-1}, C)=0
$$
for any $i>0$ and 
$C\in\scC_{n-1}\cap\Coh X$.
Then the assertion follows from (\ref{eq:check2}) and
  Lemma~\ref{lem:properties}.
 
Next let us show (\ref{eqn:projective2'}). 
By the triangle (\ref{eqn:triangle}), it is 
enough to check the following;
\begin{align}
\label{check4}
& \Hom_X^i(\Psi_{n-1}(\sigma_{\ge 1}(P_{n-1})[1]), 
 \Psi_{n-1}\circ\Phi_{n-1}(B)) =0, \\
 \label{check5}
 & \Hom_X^i(\scL^{-n},\Psi_{n-1}\circ\Phi_{n-1}(B))=0
\end{align}
for $i>0$. Note that
\begin{align}\label{eqn:projective3}
(\ref{check4})\cong 
  \Hom_{A_{n-1}}^i(\sigma_{\ge 1}(P_{n-1})[1], \Phi_{n-1}(B)).
 \end{align}
Since $\Phi_{n-1} (B)\in \module A_{n-1}$,
$\sigma_{\ge 1}(P_{n-1} )[1]\in D^b(A_{n-1})^{\ge 0}$
and each term of $\sigma_{\ge 1}(P_{n-1} )[1]$ is a projective
$A_{n-1}$-module, 
we conclude $(\ref{eqn:projective3})=0$. 
In order to check (\ref{check5}), let us
 take a free $A_{n-1}$ resolution $\scQ =(\cdots \to \scQ^{-1}
 \to \scQ^0 \to 0)$ of
an $A_{n-1}$-module $\Phi_{n-1} (B)$.
 Then each term of $\Psi_{n-1} (\scQ) $ 
is a direct sum of $\scE_{n-1}$.
Hence by Claim~\ref{cla:tilting2}, 
 we conclude (\ref{check5}) holds.
\end{proof}

We readily see that $\scE_{n-1}\in \Per0 (X/A_{n-1})$, 
and therefore we have $\scE_n\in\Per0 (X/A_{n-1})$.

\begin{cla}
$\scE_n$ is a tilting object.
\end{cla}

\begin{proof}
(\ref{check1}) yields 
\begin{align*}
\Hom_X^i(\scE_{n-1},\scN_{n-1})=0
\end{align*}
for all $i\ne 0$.
Also Claim \ref{cla:N_n} implies that
\begin{align*}
\Hom_X^i(\scN_{n-1},\scN_{n-1})&\cong\Hom_X^i(\scN_{n-1},\scE_{n-1})\\
&=0
\end{align*}
 for all $i\ne 0$.
Moreover recalling that $\scE_{n-1}$ is a tilting vector bundle,
we have
$\Hom_X^i(\scE_{n-1},\scE_{n-1})$
vanishes for $i\ne 0$. 
Combining these equalities, 
we see that $\scE_n=\scE_{n-1}\oplus \scN_{n-1}$ is a tilting object in $D^b(X)$.
\end{proof}

\begin{cla}
$\scE_n$ is a vector bundle.
\end{cla}

\begin{proof}
It is enough to show that $\scN_{n-1}$ is a vector bundle.
By Lemma \ref{lemma:vector}, we know $\Phi_{n-1}(\scO_x) \in \module A_{n-1}$
for any closed points $x\in X$, which implies 
$\scO_x \in \Per0 (X/A_{n-1})$. Hence it follows from 
Claim \ref{cla:N_n} that 
$\RHom_X(\scN_{n-1},\scO_x)\in R\module$, and in particular $\scN_{n-1}$ 
is a vector bundle by Lemma \ref{lemma:vector}.
\end{proof}

\begin{lem}\emph{\bf{\cite[Lemma 4.3]{Br2}}}\label{lemma:vector}
For a Noetherian scheme $X$ and an object $\scE\in D^b(X)$,
the following are equivalent.
\begin{enumerate}
\item
$\scE$ is a vector bundle.
\item
$\Hom_X^i(\scE, \scO_x)=0$ for any points $x\in X$ and $i\ne 0$. 
\end{enumerate}
\end{lem}

\NI
Combining the above argument and Lemma \ref{lem:generator},
we can prove the following theorem.

\begin{thm}\label{thm:main}
Let $f$ and $\scL$ be as in \S \ref{setting}. 
Assume that 
Assumption \ref{assum:induction} 
is satisfied. Then there is a vector bundle $\scE$ such that 
$\scE$ is a tilting generator of $D^-(X)$.
\end{thm}

\begin{rem}
For $n=2$, we can show that Assumption \ref{assum:induction} 
is always satisfied in \S \ref{subsection:Main result}.
 Since we have proved that $\scN_1$ is a vector 
bundle, taking the cohomology of (\ref{eqn:triangle}) yields 
$\Ext_{X}^2(\scE_1, \scL^{-2})\otimes_{A_1}\scE_1=0$. 
As $\Ext_X^2(\scE_1, \scL^{-2})$ may be non-zero, 
this vanishing is not obvious. 
\end{rem}


\section{The hearts of t-structures}\label{section:heart}
Let $f$ and $\scL$ be as in \S \ref{setting}
and furthermore assume that 
Assumption \ref{assum:induction} holds. 
Below we use the same notation as in \S \ref{section: construction},
but we omit the index $n$, for instance $\scE=\scE_n$,
$A=A_{n}=\End_X(\scE_n)$, etc. 
Recall that the equivalence
$$
\Phi=\RHom_X(\scE,-)\colon D^-(X)\longrightarrow D^-(A)
$$
induces an equivalence between $D^b(X)$ and $D^b(A)$ 
by Lemma \ref{lemma:tilting}. 

The aim of this section is the following, which will 
not be used in any subsequent sections.
Recall that $\Per0(X/A_{n-1})$ 
is, by definition, the heart of the t-structure 
$({^0\scD^{\le 0}},{^0\scD^{\ge 0}})$ 
on $D^\dag(X)$.

\begin{prop}\label{prop:hearts[0]}
The abelian category $\Per0(X/A_{n-1})$ is the heart of a 
bounded
t-structure on $D^b(X)$, and 
$\Phi(\Per0  (X/A_{n-1}))=\module A$. 
\end{prop}

\begin{proof}
We first show that $\Per0(X/A_{n-1})\subset D^b(X)$.
For an object $E\in {^0\scD^{\le 0}}$, we have the
distinguished triangle in $D^{\dag}(X)$ 
\begin{align}\label{tri3}
\Psi_{n-1}\circ\Phi_{n-1}(E) \to E \to 
i_{n-1}^{\ast}(E). 
\end{align}
By the definition of ${^0\scD^{\le 0}}$, 
we have $\Phi_{n-1}(E) \in D^b(A_{n-1})^{\le 0}$ 
and $i_{n-1}^{\ast}(E) \in \scC_{n-1}^{\le 0}$. 
Therefore $\Psi_{n-1}\circ\Phi_{n-1}(E)$ and $i_{n-1}^{\ast}(E)$
are objects in $D(X)^{\le 0}$, hence
(\ref{tri3}) yields $E\in D(X)^{\le 0}$. 
 In particular, we have $\Per0(X/A_{n-1}) \subset D^{-}(X)$. 
On the other hand, 
 Claim~\ref{cla:N_n} implies that the equivalence 
$\Phi\colon D^{-}(X) \to D^{-}(A)$ 
takes $\Per0(X/A_{n-1})$ to $\module A$. 
Since $\Phi$ restricts to an equivalence between $D^b(X)$ 
and $D^b(A)$, we must have $\Per0(X/A_{n-1})\subset D^b(X)$. 

Let $(\tau_{\le 0}^{0}, \tau_{\ge 0}^{0})$ be the truncation 
functors corresponding to the t-structure 
$({^0\scD^{\le 0}}, {^0\scD^{\ge 0}})$.
In order to conclude that $\Per0(X/A_{n-1})$ is the heart of 
a bounded t-structure of $D^b(X)$, it is enough to show that
for any object $E\in D^b(X)$, we have 
$\tau_{\le -i}^{0}(E)=\tau_{\ge i}^{0}(E)=0$ for $i\gg 0$.
 
Since the functor $\Phi_{n-1}\colon D^{\dag}(X) \to D^b(A_{n-1})$
 takes $({^0\scD^{\le 0}}, {^0\scD^{\ge 0}})$
to $(D^b(A_{n-1})^{\le 0}, D^b(A_{n-1})^{\ge 0})$, 
we have 
\begin{align}\label{eq:tau}
\Phi_{n-1}(\tau_{\ge i}^{0}(E))\cong \tau_{\ge i}^{A}(\Phi_{n-1}(E)),
\end{align}
where $(\tau_{\le 0}^{A}, \tau_{\ge 0}^{A})$
are the truncation functors with respect to the standard 
t-structure on $D^b(A_{n-1})$. 
Since $\Phi_{n-1}(E) \in D^b(A_{n-1})$, we have 
$(\ref{eq:tau})=0$ for $i\gg 0$. 
Therefore $\tau_{\ge i}^{0}(E) \in \scC_{n-1}^{\ge i} \subset D(X)^{\ge i}$. 
On the other hand, since $E\in D^b(X)$, we have 
$\Hom(E, F)=0$ for $F\in D(X)^{\ge i}$ for $i \gg 0$. 
Therefore the natural morphism $E\to \tau_{\ge i}^0(E)$ is 
zero, which implies $\tau_{\ge i}^0(E)=0$ for $i\gg 0$. 
By a similar argument, we have 
$\tau_{\le -i}^0(E)=0$ for $i\gg 0$. 

Since both of $\Phi(\Per0(X/A_{n-1}))$ and $\module A$ are  
the hearts of bounded t-structures on $D^b(A)$, and 
we also know $\Phi(\Per0(X/A_{n-1}))\subset \module A$, 
we obtain
 $$\Phi(\Per0(X/A_{n-1}))=\module A.$$ 

\end{proof}

Assume furthermore that the equality (\ref{eq:assume}) 
holds. Then Remark \ref{rem:easycase} implies that 
$\scE$ and $\scE^\vee$ are tilting generators of $D^-(X)$.  
We define the functor
$$
\Phi_k^\vee=\RHom_X(\scE_k^\vee,-)\colon D(X)\longrightarrow D(A_k^\circ),
$$
and then $\Phi^\vee=\Phi_{n}^\vee$
gives an equivalence between $D^b(X)$ and $D^b(A^\circ)$.
Here, we identify $D(A_k^\circ)$ with $D(\End_X(\scE_k^\vee))$,
using the isomorphism  $A_k^\circ\cong \End_X(\scE_k^\vee)$. 

Define the full subcategories of 
the unbounded derived category $D(X)$ as
\begin{align*}
\scC_{n-1}^\vee &=\{ \scK \in D(X) \mid \Phi_{n-1}^\vee(\scK) =0\}\\
D^{\dag\dag}(X)&=\left\{ \scK \in D(X) \bigm|
\Phi_{n-1}^\vee(\scK)\in D^b( A_{n-1}^\circ) \right\}.
\end{align*}
It is easy to see from 
Lemma \ref{lem:generator} that 
for an object $\scK\in D(X)$, 
$\scK$ belongs to $\scC^\vee_{n-1}$ 
if and only if 
$\scK\otimes \scL^{\otimes -n+1}$ belongs to $\scC_{n-1}$.
Therefore we can check that 
$\scK\in\scC_{n-1}^\vee$ implies that $\scH^k(\scK)\in\scC_{n-1}^\vee$ 
for all $k$ by Assumption \ref{assum:induction}, and hence
by the exact triple of triangulated categories 
$$
\scC_{n-1}^\vee{\to} D^{\dag\dag}(X)\stackrel{\Phi_{n-1}^\vee}{\to} D^b(A_{n-1}^\circ),
$$ 
we can define the category  
of perverse coherent sheaves 
$\Per{p}(X/A_{n-1}^\circ)$
as $\Per{p}(X/A_{n-1})$.

Note that (\ref{eqn:shortk}) yields 
$
\RHom_X (\scN_{n-1}^\vee,C)=\RHom_X (\scL^{n},C)
$
for $C\in \scC_{n-1}^\vee$,
hence Lemma \ref{lem:properties} implies 
$$
\RHom_X (\scN^{\vee}_{n-1},C)\in R\module 
$$
for $C\in \scC_{n-1}^\vee\cap \Coh X[n]$.
In particular, we see 
\begin{equation}\label{eqn:coh2}
\Phi^\vee(\scC_{n-1}^\vee\cap \Coh X[n])\subset \module A^\circ,
\end{equation}
The proof of the next proposition uses this fact.
 (By comparison,
$$
\Phi (\scC_{n-1}\cap \Coh X)\subset \module A
$$
holds by Lemma \ref{lem:properties} and (\ref{eq:check2}).
The proof of Claims \ref{cla:N_n0}
and \ref{cla:N_n} relies on this fact.)

\begin{prop}\label{prop:hearts[n]}
In the setting of Proposition \ref{prop:hearts[0]}, assume furthermore 
the equality (\ref{eq:assume}) holds.
Then the abelian category $\Per{-n}(X/A_{n-1}^\circ)$ is the heart of a 
bounded t-structure on $D^b(X)$, and 
$\Phi^\vee(\Per{-n} (X/A_{n-1}^\circ))=\module A^\circ$.
\end{prop}
\begin{proof}
We outline the proof and leave the details to the reader.
First we show that
the object 
$\scN_{n-1}^\vee$ belongs to $\Per{-n} (X/A_{n-1}^\circ)$ as Claim \ref{cla:N_n0}.
In the proof, we use (\ref{eqn:coh2}).  

Next, we mimic the proof of Claim \ref{cla:N_n} and show 
\begin{equation*}
\Hom_X ^i(\scN_{n-1}^\vee,B)=0
\end{equation*} 
for $i>0$ and $B\in\Per{-n} (X/A_{n-1}^\circ)$.
We again use (\ref{eqn:coh2}) here.

From these facts, we can conclude
$$
\Phi^\vee(\Per{-n}(X/A_{n-1}^\circ))\subset \module A^\circ
$$ 
and then a similar argument to Proposition \ref{prop:hearts[0]}
works.
\end{proof}


\begin{exa}
In this example,
we show that tilting generators induce the derived equivalence 
between certain varieties connected by a Mukai flop.  
We also apply Propositions \ref{prop:hearts[0]} and \ref{prop:hearts[n]}.
 
Let $X$ be the cotangent bundle $T^* \bP ^n$ of the projective space $\bP ^n$
$(n\ge 2)$ and $g\colon Z \to X$ a blow-up along the zero section
of the projection $\pi\colon X \to \bP ^n$.
The exceptional locus $E(\subset Z)$ of $g$ 
is the incidence variety in $\bP^n \times (\bP^n)^{\vee}$, 
where $(\bP^n)^{\vee}$ is the dual projective space. 
By contracting curves contained in fibers of the projection 
$E\to (\bP^n)^{\vee}$, 
we obtain a birational contraction $g^{+}\colon 
Z \to X^{+}=T^* ((\bP ^n)^\vee)$. 
The resulting birational map 
$$
\phi =g^{+}\circ g^{-1}\colon X \dashrightarrow X ^+
$$
is so called a \emph{Mukai flop}. Put $R=\Spec H^0(X,\mathcal{O}_X)$.
Then we have a birational contraction 
$$
f\colon X\to Y=\Spec R
$$
which contracts only the zero section of $\pi$. In particular,
$f$ has at most $n$-dimensional fibers. 

We put $\scO_X(1)=\pi^* \scO_{\bP ^n}(1)$.
Then by direct calculations (refer to calculations 
in \S \ref{section:G(2,4)}) and Lemma \ref{lem322} 
we know that 
$
\scE=\bigoplus _{i=0}^n\scO_X(-i)
$
is a tilting generator of $D^-(X)$.
On the other hand,
we can see that (\ref{eq:assume}) holds for $\scL=\scO_X(1)$.
Apply the arguments in \S \ref{inductive} and 
Remark \ref{rem:easycase}; we obtain tilting vector bundles 
$\scE_k=\bigoplus _{i=0}^k\scO_X(-i)$ for all $k$ with $0\le k\le n$
(in other words, 
$r_k=0$ in (\ref{eqn:shortk}) for all $k$ with $0< k\le n$).
We can also check that 
Assumption \ref{assum:induction} holds. 
Therefore we can apply Propositions \ref{prop:hearts[0]} and \ref{prop:hearts[n]}.  

In what follows, we use the same notation as in the previous section,
and we also use the superscript $+$ to denote the corresponding object
on $X^+$ to the object on $X$. For instance,
$
\scE^+=\bigoplus _{i=0}^n\scO_{X^+}(-i).
$

Since $\phi$ is isomorphic in codimension one, there is an equivalence 
between categories of reflexive sheaves on $X$ and $X^+$. Hence,
we have a reflexive 
sheaf $\scE'$ on $X^+$ corresponding to $\scE$, satisfying 
$\End_X(\scE)\cong\End_{X^+}(\scE')$.
It is known that 
the corresponding reflexive sheaf on $X^+$ 
to $\scO_X(-1)$ is $\scO_{X^+}(1)$ (cf. \cite[Lemma 1.3]{Na1}, 
\cite[Lemma 2.3.1]{Na2}).
From these facts, 
we see that $\scE'\cong(\scE^+)^\vee$
and so we have 
an isomorphism of rings, denoted by $\phi_*$:
$$
\phi _*\colon A=\End_X(\scE)\cong\End_X(\scE')\cong 
\End_X((\scE^+)^\vee)\cong A^\circ.
$$
In particular, we have an equivalence
$D^b(A)\cong D^b(A^\circ)$ preserving 
the hearts of the standard t-structures. 
Compose this equivalence 
with equivalences given by tilting generators $\scE$ and $(\scE^+)^\vee$, 
and then we obtain an equivalence 
$$
\Xi \colon D^b(X)\to D^b(A)\to D^b(A^\circ)\to D^b(X^+),
$$
which satisfies
$\Xi(\Per{0} (X/A_{n-1}))=\Per{-n} (X^+/A_{n-1}^\circ)$ by Propositions 
\ref{prop:hearts[0]} and \ref{prop:hearts[n]}.  
Compare the results in \cite{Na1} and \cite[Corollary 5.7]{Kaw02}, where 
a similar derived equivalence is shown 
to exist by a very different method from ours.
\end{exa}


\section{The case of two-dimensional fibers}\label{section:2dimensional}


\subsection{Main result}\label{subsection:Main result}

Let $f \colon X\to Y=\Spec R$ be a projective morphism 
from a Noetherian scheme to an affine scheme 
of finite type over a field, 
or an affine scheme of a Noetherian complete local ring.
Suppose that the fibers of $f$ are
at most two-dimensional. 
Assume furthermore that $\bR f_*\scO_X=\scO_Y$ and 
there is an ample, globally generated line bundle
$\scL $ on $X$, satisfying $\bR ^2 f_{\ast}\scL ^{-1}=0$.
\NI
The following is a main theorem in this paper. 
\begin{thm}\label{thm:rel.dim2}
Under the above situation, 
there is a tilting vector bundle generating
the derived category $D^-(X)$.
\end{thm}
\begin{proof}
We have to show that 
Assumption~\ref{assum:induction} holds so that we apply Theorem
\ref{thm:main}.
Take $\scK \in D(X)$, which satisfies
\begin{align}\label{eqn:induction1}
\bR f_{\ast}\scK = \bR f_{\ast}(\scK \otimes \scL)=0.
\end{align}
Let $H\in \lvert \scL \rvert$ be a general member. 
In what follows, we repeatedly use the fact that
 $\scH^k(\scK|_{H})=\scH^k(\scK)|_{H}$ 
for any $k\in \mathbb{Z}$, since $H$ is a general
member. 
We have 
the distinguished triangle
\begin{align*}
\scK \to \scK \otimes \scL \to \scK|_{H}\otimes \scL.
\end{align*}
Applying $\bR f_{\ast}$ and using (\ref{eqn:induction1}), 
we obtain 
$\bR f_{\ast}(\scK|_{H}\otimes \scL)=0$.
Since $f|_{H}\colon H\to f(H)$ has at most one-dimensional 
fibers, we have (cf. \cite[Lemma 3.1]{Br1})
\begin{align}\label{eqn:induction2}
\bR f_{\ast}(\scH^k(\scK|_{H}\otimes \scL))=0
\end{align}
for any $k$. Similarly, applying $\bR f_{\ast}$ to the triangle 
$$
\scH^{k}(\scK) \to \scH^{k}(\scK\otimes \scL) \to 
\scH^{k}(\scK|_{H}\otimes \scL)
$$
and using (\ref{eqn:induction2}), we obtain 
\begin{align}\label{eqn:induction3}
\bR f_{\ast}(\scH^{k}(\scK))
\cong
\bR f_{\ast}(\scH^{k}(\scK\otimes \scL)).
\end{align}
Next let us consider the spectral sequence:
$$
E_2^{p,q}=\bR^p f_{\ast}(\scH^q(\scK)) \Rightarrow 
\bR^{p+q}f_{\ast}\scK.
$$
Since $E_2^{p,q}=0$ unless $0\le p\le 2$, the above spectral 
sequence and (\ref{eqn:induction1}) imply
\begin{align}\label{eqn:induction4}
\bR^1 f_{\ast}(\scH^k(\scK))=0, \quad 
f_{\ast}(\scH^{k+1}(\scK)) 
\cong
\bR^2 f_{\ast}(\scH^{k}(\scK))
\end{align}
for any $k$. By (\ref{eqn:induction3}) and (\ref{eqn:induction4}),
if we show
$\bR^2 f_{\ast}(\scH^{k}(\scK))=0$ for any $k$,  
then the conclusion of Assumption~\ref{assum:induction} follows.

 Suppose that $\bR^2 f_{\ast}(\scH^{k}(\scK))\neq 0$ for some $k$. 
By the formal function theorem, there is a closed 
sub-scheme $E\subset X$ supported by a two-dimensional fiber of $f$, such that 
 $H^2(E, \scH^{k}(\scK)|_{E})\neq 0$. 
By the Grothendieck duality, we have 
$$
0\neq H^2(E, \scH^{k}(\scK)|_{E}) \cong 
\Hom_E(\scH^{k}(\scK)|_{E}, \scH^{-2}(D_E))^{\vee}.
$$
Let $u\colon \scH^{k}(\scK)|_{E} \to \scH^{-2}(D_E)$
be a non-zero morphism, and consider its image 
$\Image u \subset \scH^{-2}(D_E)$. Then the support of $\Image u$ 
is two-dimensional because 
$$
0\neq \Hom_E(\Image u, \scH^{-2}(D_E)) \cong 
H^2(E, \Image u)^{\vee}
$$
by the duality.
Hence by the choice of $H\in \lvert \scL \rvert$, we may assume that 
$(\Image u)|_{H} \neq 0$. We may also assume that
$H$ does not contain any associated prime of $\Coker u$.
Then we can show that
$u|_{H} \colon 
\scH^{k}(\scK)|_{E\cap H} \to \scH^{-2}(D_E)|_{H}$
is a non-zero morphism. By adjunction, we have 
$$
D _{H\cap E}
\cong (D_E[-1]\otimes \scL)|_{H}.
$$
Hence $u|_{H}$ induces the non-zero morphism in
$$
\Hom_{E\cap H}(\scH^{k}(\scK)|_{E\cap H},
\scH^{-1}(D_{H\cap E}\otimes \scL^{-1})).$$
Then the duality on $E\cap H$ implies
\begin{align}\label{eqn:induction5}
0\neq 
H^1(E\cap H, \scH^{k}(\scK)|_{E\cap H}\otimes \scL) \cong
H^1(E\cap H,\scH^{k}(\scK|_{H}\otimes \scL)|_{E}).
\end{align}
On the other hand, the surjection 
$$
\scH^k(\scK|_{H}\otimes \scL) 
\twoheadrightarrow \scH^{k}(\scK|_{H}\otimes \scL)|_{E}
$$
induces the surjection 
$$
\mathbb{R}^1 f_{\ast}(\scH^k(\scK|_{H}\otimes \scL))
\twoheadrightarrow \mathbb{R}^1 f_{\ast}(\scH^k(\scK|_{H}\otimes \scL)|_{E}).
$$
However this contradicts (\ref{eqn:induction2}) and 
(\ref{eqn:induction5}), 
hence it follows that $\bR^2 f_{\ast}(\scH^{k}(\scK))=0$. 
\end{proof}


\subsection{Crepant resolutions of three dimensional canonical 
singularities}
Let $0\in Y=\Spec R$ be a $3$-dimensional canonical 
singularity and $R$ be a Noetherian complete local ring. 
Suppose that there is a crepant resolution $f\colon X\to Y$
such that the exceptional locus is an irreducible divisor 
$E\subset X$ and $\mathbb{R}f_*\scO_X=\scO_Y$. 
Then, $E$ is a generalized del Pezzo surface:
 that is, $\omega _E^{-1}$ is ample.
We aim to construct a tilting generator 
of $D^-(X)$. 

\begin{lem}\label{lem:extension}
If there is an ample, globally generated line bundle $\scL_1$ on $E$ with 
$H^2(E, \scL _1^{-1})=0$, then we have an 
ample, globally generated  line bundle 
$\mathcal{L}$ on $X$ 
such that 
$\mathbb{R}^2 f_{\ast}\mathcal{L}^{-1}=0$. 
\end{lem}

\begin{proof}
Let $I_E \subset \mathcal{O}_X$ be the defining ideal of 
$E$ and $E_n \subset X$ the subscheme defined by $I_E ^n$ for $n>0$. 
Then the obstruction to extend a line bundle $\mathcal{L}_n \in \Pic(E_n)$ 
to a line bundle $\mathcal{L}_{n+1}\in \Pic(E_{n+1})$ lies in 
$H^2 (E, I_E ^n/I_E ^{n+1})$. We have 
\begin{align*}
H^2 (E, I_E ^n/I_E ^{n+1}) & \cong H^2 (E, \mathcal{O}_E(-nE)) \\
& \cong H^0(E, \mathcal{O}_E((n+1)E))^{\vee} \\
& = 0.
\end{align*}
Here, the second isomorphism follows from the Serre duality, and the last 
isomorphism holds because $-E$ is $f$-ample. Hence, for a
given line bundle $\scL _1 \in \Pic(E)$, we obtain an element 
$$\hat{\mathcal{L}}=\{ \mathcal{L}_n \}_{n\ge 1}
 \in \lim _{\longleftarrow}\Pic(E_n) \cong \Pic(\hat{X}).$$
By the Grothendieck existence theorem, 
there is a line bundle $\mathcal{L}$ on $X$ such that 
$\mathcal{L}|_{\hat{X}}\cong \hat{\mathcal{L}}$. 

Take an ample, globally generated line bundle $\scL _1$ on $E$ such that 
$H^2 (E, \scL _1^{-1})=0$. Let 
 $\mathcal{L} \in \Pic(X)$ be its extension. 
We have
\begin{align*}
H^2 (E, \scL _1^{-1}\otimes I_E ^n /I_E ^{n+1}) & \cong 
H^0(E, \scL _1\otimes \mathcal{O}_E (nE)\otimes \omega _E)^{\vee}\\
&= 0,
\end{align*}
since $-E$ is $f$-ample and $H^0(E, \scL _1\otimes \omega _E)=0$. 
Hence $\mathbb{R}^2 f_{\ast}\mathcal{L}^{-1}=0$
by the formal function theorem. 
$\mathcal{L}$ is also globally generated by the basepoint free theorem, 
and clearly $\mathcal{L}$ is ample.
\end{proof}

 In particular, Theorem \ref{thm:rel.dim2} implies the following.

\begin{thm}\label{thm:CY}
In the situation of Lemma \ref{lem:extension},
there is a tilting generator of $D^-(X)$.
\end{thm}

\begin{exa}
There is a $3$-dimensional crepant resolution $f\colon X\to Y$ 
from a Calabi-Yau threefold $X$ defined over $\mathbb{C}$
whose exceptional locus is isomorphic to $E$ in (i), (ii) below
(\cite{Kapu2}, \cite{Kapu}).
Replace $Y$ with its completion at the singular point 
and shrink $X$ accordingly.

We show the existence of 
tilting generators of $D^-(X)$.
The key fact is that if we have a line bundle $\scL_1$ on $E$, as in  
Lemma \ref{lem:extension}, then we can find a tilting generator of $D^-(X)$
by Theorem \ref{thm:CY}.
\begin{enumerate}
\item
 The first example is a quadric $E\subset \mathbb{P}^3$, 
 that is, $E$ is isomorphic to $\mathbb{P}^1 \times \mathbb{P}^1$
or the cone over a conic.  
Then $\scL _1=\mathcal{O}_{\mathbb{P}^3}(1)|_{E}$
satisfies $H^2 (E, \scL_1^{-1})=0$. 
\item 
For the second example, take the cone over a conic $S \subset \mathbb{P}^3$.
Let $E$ be a surface obtained by the
blowing-up $\pi \colon E\to S$ at a non-singular point in $S$. 
Note that $E$ is a singular del Pezzo surface. 
Denote by $C$ the exceptional curve of $\pi$ and
put $\mathcal{O}(1)=\mathcal{O}_{\bP ^3}(1)|_S$. 
 Then $\scL _1=\pi^{\ast}\mathcal{O}(1) \otimes \mathcal{O}_E(-C)$
is an ample, globally generated line bundle satisfying 
$H^2(E,\scL_1^{-1})=0$. 


\end{enumerate}
\end{exa}


\section{The cotangent bundle of $G(2,4)$}\label{section:G(2,4)}

In \S \ref{sec:preliminary}, 
we cite and prove some results that \S \ref{sub1} uses.
In \S \ref{sub1}, 
we find tilting generators on a
one-parameter deformation of the cotangent bundle $X_0=T^*G(2,4)$,
where $G(2, 4)$ is 
the Grassmann manifold. 
We assume all varieties are defined over $\bC$ in this section.


\subsection{The Bott theorem}\label{sec:preliminary}
Let $G$ be the Grassmann manifold $G(k,V)$ of $k$-dimensional subspaces
in an $n$-dimensional $\bC$-vector space $V$. 
There is a non-split exact sequence
$$
0\to \Omega _G\to \widetilde{\Omega}_G \to \scO_G \to 0
$$
corresponding to a nonzero element of the $1$-dimensional space
$H^ 1(G,\Omega _G)$. Put 
$\widetilde{T}_G=(\widetilde{\Omega}_G)^\vee$.
We denote the total space of 
$\widetilde{\Omega}_G$ (resp. $\Omega_G$)
by $X$ (resp. $X_0$).
Then 
there is a one-parameter deformation (\cite{Na2}, \cite{Kaw24})
$$
\xymatrix{
X \ar[r]^{f}\ar[dr] & Y \ar[d] \\
& \mathbb{A}^1
}
$$
of the Springer resolution 
$$
f_0\colon X_0\to Y_0=\Spec R_0.
$$ 
We denote by $\pi\colon X \to G$ and  
$\pi_0\colon X_0\to G$ the projections.

Let $\scU$ be 
the tautological $k$-dimensional sub-bundle of $\scO_G\otimes V$.
We also define $\scU^\perp$ to be 
$((\scO_G\otimes V)/ \scU)^\vee$,
the dual of the quotient bundle.
For a vector bundle $\scE$ of rank $m$ on $G$,
we consider the associated principal $\GL (m,\bC)$-bundle and
denote by $\Sigma ^\alpha \scE$ the vector bundle 
associated with the $\GL (m,\bC)$ representation
of highest weight $\alpha\in \bZ^m$.
 For $\alpha=(\alpha_1,\ldots,\alpha_m)\in \bZ^m$ with
$\alpha_1\ge\cdots\ge \alpha_m$ 
(such a sequence is called a \textit{non-increasing} sequence),
we have 
$$
\Sigma ^\alpha (\scE^\vee) =
\Sigma ^{(-\alpha_m,\ldots,-\alpha_1)} \scE= 
(\Sigma ^\alpha \scE)^\vee.
$$ 

We have the following equality:

\begin{align}\label{eqn:base2}
&\Hom^i_G(\Sigma^{\alpha}\scU,\Sigma^{\beta}\scU
\otimes (\bigoplus_{n\ge 0}\Sym ^n(T_G)))\notag\\
=&\bigoplus_{n\ge 0} \bigoplus_{|\lambda|=n}
H^i(G,\Sigma^{\alpha}\scU^\vee\otimes\Sigma^{\beta}\scU
\otimes
\Sigma^{\lambda}\scU^\vee\otimes \Sigma ^{\lambda}(\scU^{\perp})^\vee). 
\end{align}
%
Here $|\lambda|=\sum \lambda_l$ and all the $\lambda_l$'s are non-negative.
For the proof of (\ref{eqn:base2}), see \cite[page 80]{F-H} and use 
$T_G=\scU^\vee\otimes(\scU^{\perp})^\vee$.

\begin{lem}\label{lem:X_0toX}
Suppose that
the vector space in (\ref{eqn:base2}) is 0-dimensional
for fixed $i$, $\alpha$ and $\beta$. 
Then the vector space
$
\Hom^i_X(\pi^*\Sigma^{\alpha}\scU,\pi^*\Sigma^{\beta}\scU)
$
is also 0-dimensional.
\end{lem}

\begin{proof}
The assertion follows from the equality
\begin{align*}
\Hom^i_X(\pi^*\Sigma^{\alpha}\scU,\pi^*\Sigma^{\beta}\scU)
&\cong \Hom^i_G(\Sigma^{\alpha}\scU,\Sigma^{\beta}\scU\otimes \pi_*\scO_X) \\
&\cong \Hom^i_G(\Sigma^{\alpha}\scU,\Sigma^{\beta}\scU
\otimes (\bigoplus_{n\ge 0}\Sym ^n(\widetilde{T}_G)))
\end{align*}
and the filtration
$$
\Sym^n(\widetilde{T}_G)=F^0\supset F^1\supset \cdots 
\supset F^n\supset F^{n+1}=0 
$$
with $F^l/F^{l+1}\cong \Sym^{n-l}(T_G)$.
\end{proof}

Let $F(V)$ be the flag variety of $\GL (V)$ and
$$
\scU_1\subset \scU_2\subset\cdots\subset \scU_{n-1}\subset 
\scU_{n}=
V\otimes \scO_{F(V)}.
$$ 
the sequence of the universal sub-bundles $\scU_i$ of rank $i$.
We put 
$$
\scO(\delta_1,\ldots,\delta_n)=
\scU_1^{-\delta_1}
\otimes(\scU_2/\scU_1)^{-\delta_2}\otimes 
\cdots \otimes( \scU_{n}/ \scU_{n-1})^{-\delta_n}.
$$ 
\NI
The following lemma is taken from the proof of
 \cite[Proposition 2.2]{Kapranov}.
\begin{lem}\label{lem:Kapranov}
For non-increasing sequences $\alpha\in \bZ^k,\beta \in \bZ^{n-k}$, 
we have 
$$
H^i(G,\Sigma^{\alpha}\scU^\vee\otimes\Sigma^{\beta}\scU^\perp)
=H^i(F(V),\scO(\Delta)),
$$
where
$\Delta=(\alpha_1,\ldots,\alpha_k,\beta_1,\ldots,\beta_{n-k})$.
\end{lem}
\NI
By Lemma \ref{lem:X_0toX} and Lemma \ref{lem:Kapranov},
showing the vanishing of the vector space 
$$
\Hom^i_X(\pi^*\Sigma^{\alpha}\scU,\pi^*\Sigma^{\beta}\scU)
$$
is reduced to the dimension counting of 
the cohomology $H^i(F(V),\scO(\Delta))$ on the flag variety $F(V)$.
Hence, we shall compute $H^i(F(V),\scO(\Delta))$ 
for 
$\Delta=(\delta_1,\ldots,\delta_n)\in \bZ^n$.
The permutation group $\mathfrak S _n$ naturally acts on $\bZ^n$:
$$
\sigma (\delta_1,\ldots,\delta_n)
=(\delta_{\sigma(1)},\ldots,\delta_{\sigma(n)}).
$$
We also define the tilde action of $\mathfrak S _n$ on $\bZ^n$:
$$
\tilde{\sigma} (\Delta)=\sigma(\Delta+\rho)-\rho.
$$
Here $\rho =(n-1,n-2,\ldots,0)$. For instance, 
when we put $\sigma_l=(l\ l+1)$,  we obtain
$$
\tilde{\sigma} _l(\delta_1,\ldots,\delta_n)
=(\delta_1,\ldots,\delta_{l-1},\delta_{l+1}-1,\delta_l+1,\delta_{l+2},
\ldots,\delta_n).
$$
The Bott theorem implies that:

(1) If $\Delta$ is non-increasing, then we have
$$
H^i(F(V),\scO(\Delta))=
\begin{cases}
\Sigma^\Delta V & i=0\\
0 &i>0.
\end{cases}
$$

(2) If $\Delta$ is not non-increasing, then
we apply the tilde action of $\frak S_n$ 
for transpositions like $\sigma_l=(l\ l+1)$, 
trying to move bigger numbers to the 
right past smaller numbers. Repeat this process.
Then there are two possibilities:
\begin{itemize}
\item
Suppose that eventually, 
we achieve $\delta_{l+1}=\delta_l+1$ 
for some $l$.
Then $H^i(F(V),\scO(\Delta))=0$ for all $i$.
\item
Suppose that after applying $j$ times tilde actions 
of transpositions in $\frak S_n$,
we can transform $\Delta$ into a non-increasing sequence $\Delta _0$.
Then we have
$$
H^i(F(V),\scO(\Delta))=
\begin{cases}
\Sigma^{\Delta _0} V & i=j\\
0 &i\ne j.
\end{cases}
$$
\end{itemize}

\subsection{$G(2,4)$}\label{sub1}
Henceforth in this section, $G$ denotes $G(2,4)$.
Let us find a tilting generator of $D^-(X)$ 
using Theorem \ref{thm:main} in this subsection.
Let $\scO_G(1)=\Sigma^{(-1,-1)}\scU$ be a line bundle on $G$ 
which gives the Pl\"ucker embedding 
$G\hookrightarrow \bP^5$ and we denote $\pi^*\scO_G(1)$ 
by $\scO_X(1)$.

First we want to show (\ref{eqn:ample2}) for $\scL=\scO_X(1)$;
namely 
\begin{equation}\label{eqn:Xample2}
H^i(X,\scO_X(-j))(=
\Hom^i_X(\pi^*\Sigma^{(0,0)}\scU,\pi^*\Sigma^{(j,j)}\scU))
=
0
\end{equation}
for $0<j<4$ and $i\ge 2$.
Putting $\alpha=(0,0)$ and $\beta=(j,j)$ in (\ref{eqn:base2})
and using Lemma \ref{lem:Kapranov}, we obtain
\begin{align}\label{eqn:Xample2'}
&\Hom^i_G(\Sigma^{(0,0)}\scU,\Sigma^{(j,j)}\scU
\otimes (\bigoplus_{n\ge 0}\Sym ^n(T_G)))\notag\\
=&\bigoplus_{n\ge 0} \bigoplus_{|\lambda|=n}
H^i(G,
\Sigma^{(j-\lambda_2,j-\lambda_1)}\scU\otimes 
\Sigma^{(-\lambda_2,-\lambda_1)}\scU^\perp)\notag\\
=&\bigoplus_{n\ge 0} \bigoplus_{|\lambda|=n}
H^i(F(V),\scO(\lambda_1-j,\lambda_2-j,-\lambda_2,-\lambda_1)),
\end{align}
where we put $\lambda =(\lambda_1,\lambda_2)$.
For the proof of (\ref{eqn:Xample2}), by Lemma \ref{lem:X_0toX}, 
it is enough to see the vanishing of (\ref{eqn:Xample2'})
for $0<j<4$ and $i\ge 2$.

Denote 
$$
\Delta=(\lambda_1-j,\lambda_2-j,-\lambda_2,-\lambda_1)
$$ 
below.
The Bott theorem says that one of the following occurs:
\begin{itemize}
\item
If $\lambda_2-j \ge -\lambda_2$, then  
$
H^i(F(V),\scO(\Delta))= 0
$
if and only if $i\ne 0$.
\item
If $\lambda_2-j+1= -\lambda_2$, then
$
H^i(F(V),\scO(\Delta))= 0
$
for all $i$.

\item 
If $\lambda_2-j+1< -\lambda_2$, then $\lambda_2=0$ and $j=2,3$,
which implies 
$$
\tilde{\sigma}_2\Delta=(\lambda_1-j,-1,-j+1,-\lambda_1).
$$
In the case $\lambda_1-j\ge -1$, 
$
H^i(F(V),\scO(\Delta))\ne 0
$
implies $i= 1$. 
In the case $\lambda_1-j+1=-1$, 
$
H^i(F(V),\scO(\Delta))= 0
$
for all $i$. In the case $\lambda_1-j+1<-1$, we obtain
$
\lambda_1=0
$ 
and 
$j=3$. Then it is easy to see that 
$
H^i(F(V),\scO(\Delta))= 0
$
for all $i$.
\end{itemize}

\NI
Therefore we obtain 
(\ref{eqn:Xample2}) as desired.

Next we want to check that Assumption \ref{assum:induction} is true, i.e.
$\scK\in D(X)$ satisfies the equality
\NI
\begin{equation}\label{eqn:241}
\bR f_{*}(\scH^k(\scK)\otimes \scO_X(j))=0
\end{equation}
\NI
for any $k$ and $j$, $(0\le j\le 3)$ when we assume the equalities 
\begin{equation}\label{eqn:242}
\bR f_{*}(\scK\otimes \scO_X(j))=0
\end{equation}
for any $j$, $(0\le j \le 3)$.
Because $\scO_X(1)$ gives an embedding 
$h\colon X\hookrightarrow \bP^5_R,$
we can say that (\ref{eqn:242}) is equivalent to 
\NI
\begin{equation}\label{eqn:243}
\bR g_{*}(h_{*}\scK\otimes \scO(j))=0
\end{equation}
for all $j$ with $0\le j\le 3$,
where $g \colon \bP_R^5 \to \Spec R$ is the structure morphism.
On the other hand, $D(\bP_R^5)$ has a semi-orthogonal decomposition
$$
D(\bP_R^5)=\Span{g^{*}D(R)\otimes \scO_{\bP_R^5}(-5), 
g^{*}D(R)\otimes \scO_{\bP_R^5}(-4),\ldots, g^{*}D(R)},
$$
and hence
it follows from our assumption (\ref{eqn:243}) that 
$$
h_{*}\scK \in \Span{g^{*}D(R)\otimes \scO_{\bP_R^5}(-5), g^{*}D(R)
\otimes \scO_{\bP_R^5}(-4)}.
$$ 
Consequently, there is a triangle
\NI
$$
\cdots\to g^{*}W_{-4} \otimes _R \scO_{\bP_R^5}(-4) 
\to h_{*}\scK \to g^{*}W_{-5} \otimes _R \scO_{\bP_R^5}(-5)
\to \cdots
$$
for some $W_l \in D(R)$, and then 
we obtain a long exact sequence
\NI
$$
\cdots\to \scH^k(W_{-4})\otimes _R \scO_{\bP_R^5}(-4) \to \scH^k(h_{*}\scK)
\to \scH^k (W_{-5})\otimes _R \scO_{\bP_R^5}(-5) \to \cdots.
$$
Because the support of $\scH^k(h_{*}\scK)$ is contained in $X$ and 
the support of $\scH^k(W_{-5})\otimes _R \scO_{\bP_R^5}(-5)$ is 
the inverse image of some closed subset on $Y$ by $g$,
the morphism 
$\scH^k(h_{*}\scK) \to \scH^k (W_{-5})\otimes _R \scO_{\bP_R^5}(-5)$
should be zero. Therefore we have a short exact sequence
$$
0 \to \scH^{k-1}(W_{-5})\otimes_R \scO_{\bP_R^5}(-5) 
\to \scH^{k}(W_{-4})\otimes_R \scO_{\bP_R^5}(-4)
\to \scH^k(h_{*}\scK) \to 0.$$
Then (\ref{eqn:241}) follows.
Now we can construct a tilting generator of $D^-(X)$ 
by Theorem \ref{thm:main}.

We have proved the following:

\begin{thm}\label{thm:G(2,4)}
The derived category $D^-(X)$ 
has a tilting generator 
which is a vector bundle on $X$. 
\end{thm}

\begin{cor}[cf. \cite{Kale}]\label{cor:g(2,4)}
The derived category $D^-(X_0)$ 
has a tilting generator 
which is a vector bundle on $X_0$. 
\end{cor}

\begin{proof}
Let $\scE$ be a tilting generator in $D^-(X)$ constructed above.
Put $\scE_0=i^*\scE$, where $i\colon X_0\hookrightarrow X$ is the embedding.
Since $X$ is a one-parameter deformation of $X_0$, 
there is an exact sequence
$0\to \scO_X\to \scO_X \to \scO_{X_0}\to 0$. 
Taking a tensor product with $\scE$, we obtain an exact sequence
\begin{align}\label{eqn:seq}
0\to \scE \to \scE \to \scE_{0}\to 0.
\end{align}
Applying $\RHom(\scE, -)$ to (\ref{eqn:seq}), 
we can conclude that $\scE_0$ is a tilting object.
We can directly check that $\scE_0$ is a generator. 
\end{proof}



\section{Auxiliary result:
the existence of a right adjoint functor}\label{section:auxiliary}

In this section, we show the existence of a right adjoint functor
of $\Phi_{n-1}$, which is needed in \S \ref{subsection:gluing}.
Let $Y$ be a scheme of finite type over a field or 
a spectrum of a Noetherian complete local ring.
This condition assures the existence of the dualizing complex
on $Y$.
Let us consider a projective morphism between schemes $f\colon X\to Y$.
Then we know that $R=H^0(X,\scO_X)$ has the dualizing complex $D_R$.
For a vector bundle $\scE$ on $X$,
put 
\begin{align*}
&\scA=\scEnd_X\scE,\quad A=\End_X \scE, \\
&D_\scA=\scRHom_X(\scA,D_X), \quad D_A=\RHom_R(A,D_R),\\
&\bD_\scA(-)=\scRHom_\scA(-,D_\scA)\colon D^-(\scA)\to D^+(\scA^{\circ}),\\
&\bD_A(-)=\RHom_A(-,D_A)\colon D^-(A)\to D^+(A^{\circ}),\\
&\tilde{\Phi}(-)=\scRHom_X(\scE,-)\colon D^-(X)\to D^-(\scA),\\
&\Phi(-)=\RHom_X(\scE,-)\colon D^-(X)\to D^-(A),\\
&\Psi(-)=(-)\Lotimes_A \scE \colon D^{-}(A) \to D^{-}(X),\\
&\bD_R=\RHom_R(-,D_R) \colon D^-(R)\to D^+(R).
\end{align*}
For the dual vector bundle $\scE^\vee$  of $\scE$, we put
\begin{align*}
 &\tilde{\Phi}^{\circ}=\scRHom_X(\scE^\vee,-)
 \colon D^+(X)\to D^+(\scA^{\circ}).
\end{align*}
Lemma \ref{lemma:duality} must be well-known to specialists.
When $\scE=\scO_X$, the lemma is a paraphrase of 
the Grothendieck duality for the natural projective morphism 
$g\colon X \to \Spec R$. 
%
\begin{lem}\label{lemma:duality}
$\bD_A\circ \Phi\cong \Phi\circ \bD_X$.
\end{lem}
\begin{proof}
We have a diagram:
\begin{equation}\label{eqn:cd}
\begin{CD}
D^-(X) @>{\tilde{\Phi}}>>    D^-(\scA)      
@>{\RGamma}>>  D^-(A) \\
@V{\bD_X}VV    @V{\bD_\scA}VV   @V{\bD_A}VV \\
D^+(X) @>{\tilde{\Phi}^{\circ}}>> D^+(\scA^{\circ})
@>{\RGamma}>>  D^+(A^{\circ}).
\end{CD}
\end{equation}
We note that there is an isomorphism 
$\Phi\cong \RGamma\circ \tilde{\Phi}$ and that $\tilde{\Phi}$ gives 
an equivalence of derived categories (\cite{Rickard}).

First, we show that the left diagram in (\ref{eqn:cd}) is commutative.
For $\scN\in D^-(X)$,
we have 
\begin{align}
\bD_\scA\circ\tilde{\Phi}(\scN)
&\cong \scRHom_\scA(\scRHom_X(\scE,\scN),D_\scA)\notag\\
&\cong 
\scRHom_\scA(\scRHom_X(\scE,\scN),\scRHom_X(\scE,\scE\otimes D_X))\notag\\
&\cong \scRHom_X(\scN,\scE\otimes D_X) \label{eqn:Ecirc1} \\
&\cong \scRHom_X(\scE^{\vee}, \scRHom_X(\scN, D_X)) \notag\\
&\cong \tilde{\Phi}^{\circ}\circ \bD_X(\scN)\notag.
\end{align}
Here, the isomorphism (\ref{eqn:Ecirc1}) comes from the Morita equivalence
$\Coh U\cong \Coh \scA|_U$ on every affine open set $U\subset X$.

Therefore, it remains to show that the right 
diagram in (\ref{eqn:cd}) is commutative.
The Grothendieck duality for $g\colon X\to \Spec R$ implies 
$$
\RGamma(D_{\scA}) \cong \RHom_{R}(\RGamma(\scA), D_{R}).
$$
Composing this isomorphism with the natural morphism
$A \to \RGamma(\scA)$, we obtain the morphism
\begin{equation}\label{eqn:omega}
\RGamma(D_{\scA})\to D_A.
\end{equation}
Moreover 
since we have
$$\Hom_A(M, \Hom_R(A, N))\cong \Hom_R(M, N)$$
for any $M\in \module A$, $N\in R\module$, 
we have the isomorphism, 
\begin{equation}\label{eqn:AR}
\RHom_A(M,D_A)\cong \RHom_R(M,D_R)
\end{equation}
in $D^-(R)$ for $M\in D^-(A)$.

For $\scM \in D^{-}(\scA)$, we have the following 
sequence of isomorphisms and natural morphisms,
\begin{align}
\RGamma\circ \bD_{\scA}(\scM)
&=\RGamma\circ \scRHom _\scA(\scM,D_{\scA}) \notag \\
&\cong \RHom _\scA(\scM,D_{\scA}) \notag \\
&\to 
\RHom_\scA(\tilde{\Phi}\circ\Psi\circ \RGamma(\scM),D_{\scA}) \label{eqn:nat}
 \\
& \cong \RHom_{A}(\RGamma(\scM), \RGamma(D_{\scA})) \label{eqn:nat2}\\ 
& \to \RHom_{A}(\RGamma(\scM), D_{A}) \label{eqn:nat3} \\
&= \bD_A \circ \RGamma(\scM) \notag.
\end{align} 
Here the morphism (\ref{eqn:nat}) and the isomorphism 
(\ref{eqn:nat2}) are obtained from the fact that
 $\tilde{\Phi}\circ\Psi$
is a left adjoint functor of $\RGamma$, and moreover the morphism
(\ref{eqn:nat3}) comes from the morphism 
(\ref{eqn:omega}). 
Consequently we obtain a morphism of functors
$$
\phi\colon \RGamma\circ \bD_{\scA}\to \bD_A\circ\RGamma.
$$

Next we want to check that $\phi$ is an isomorphism. Note that 
it is enough to check that $\phi$ is isomorphic
after applying the forgetful functor $D^-(A)\to D^-(R)$.
Take $\scN\in D^-(X)$ such that $\tilde{\Phi}(\scN)=\scM$.
Then, because of the commutativity of the left diagram in (\ref{eqn:cd}),
we have 
$$
\RGamma\circ \bD_{\scA}(\scM)\cong
\RGamma\circ \bD_X(\scE^\vee\otimes \scN).
$$
 We also have 
\begin{align*}
\bD_A\circ\RGamma(\scM)
&\cong \RHom_A(\RHom_X(\scE,\scN),D_A)\\
&\cong \RHom_R(\RHom_X(\scE,\scN),D_R)\\
&\cong \bD_R\circ \RGamma (\scE^\vee\otimes \scN).
\end{align*}
by (\ref{eqn:AR}).
Then the Grothendieck duality 
for $g$ implies that $\phi$ is isomorphic. 

\end{proof}

Put
$$D^{\dag}(X)= \left\{ 
\scK \in D(X) \bigm|
\Phi(\scK)\in D^b(A) \right\}.
$$
\begin{lem}\label{lem:adjoint}
The functor $\Phi:D^{\dag}(X)\to D^b(A)$ 
has a right adjoint functor.
\end{lem}

\begin{proof}
Indeed, using $\bD_A\circ \Phi\cong \Phi\circ \bD_X$, 
we can readily check that 
$\bD_X\circ\Psi \circ \bD_A$ is a right adjoint functor of $\Phi$.
\end{proof}


\appendix

\section{Non-commutative crepant resolution}\label{section:appendix}
First, let us recall the definition of non-commutative crepant
 resolutions introduced by Van den Bergh \cite{nonc}.

\begin{defn}
Let $k$ be a field, $R$ a normal Gorenstein finitely generated 
$k$-domain. Furthermore we denote by $A$ an $R$-algebra that 
is finitely generated as an $R$-module. $A$ is called a
non-commutative crepant resolution of $R$ if the following conditions hold:
\begin{enumerate}
\item
There is a reflexive $R$-module $E$ such that $A=\End_R(E)$.
\item
The global dimension of $A$ is finite.
\item
$A$ is a Cohen-Macaulay $R$-module. 
\end{enumerate}    
\end{defn}

\noindent
The next assertion is essentially shown in \cite{nonc}.

\begin{prop}
Let $Y=\Spec R$ be an affine normal Gorenstein variety and
assume that there is a crepant resolution 
$f\colon X\to Y$: that is, $f$ is a birational projective morphism 
from a smooth variety $X$ and $f^*\omega_Y=\omega_X$.
If we have a tilting generator $\scE$ of $D^-(X)$ such that 
$$
\Hom_X^i(\scE,\scO_X)=\Hom_X^i(\scO_X,\scE)=0
$$ 
for $i\ne0$,
then $R$ has a non-commutative crepant resolution.
\end{prop}

\begin{proof}
When $\dim R\le 1$, then $R$ is itself a non-commutative resolution of $R$. 
Thus, we assume that $\dim R\ge 2$ in what follows.
We define as 
\begin{align*}
E&=\bR\Gamma(\scE)(\cong\bR^0\Gamma(\scE)),
\quad \scA=\scRHom_X(\scE,\scE),\\
A&=\RGamma(\scA)(\cong\RHom_X(\scE,\scE)\cong \Hom_X(\scE,\scE)).
\end{align*} 
By $f^*\omega_Y=\omega_X$, we have $f^!\scO_Y=\scO_X$.
Then the Grothendieck duality and our assumptions imply that 
\begin{align*}
\Hom^i_R(E,R)
&\cong\Hom^i_X(\scE,f^!\scO_Y)\\
&\cong\Hom^i_X(\scE,\scO_X)\\
&=0
\end{align*}
for any $i\ne 0$, which implies that $E$ is Cohen-Macaulay.
We can show similarly that $A$ is Cohen-Macaulay,
since 
\begin{align*}
\Hom^i_X(\scA,\scO_X)
&\cong\Hom_X^i(\scO_X,\scA)\\
&\cong\Hom_X^i(\scE,\scE)=0
\end{align*}
for any $i\ne 0$. Note that $\End_R(E)$ and $A$ are reflexive,
since they are Cohen-Macaulay and $\dim R\ge 2$. 
Then the natural homomorphism $A\to \End_R(E)$ is isomorphic 
in codimension one, as well as everywhere else.
Moreover, $D^b(A)$ and $D^b(X)$ are derived equivalent, and therefore
the global dimension of $A$ is finite.  
\end{proof}

\begin{cor}
Let $Y=\Spec R$ be an affine normal Gorenstein variety defined 
over $\mathbb{C}$, and
suppose that there is a crepant resolution $f\colon X\to Y$
with at most two-dimensional fibers.
Further assume that we have a globally generated,
ample line bundle $\scL$ on $X$ which satisfies 
$\bR^2 f_{\ast}\scL^{-1}=0$. 
Then $R$ has a non-commutative crepant resolution.
\end{cor}

\begin{proof}
Note that $\bR f_*\scO_X\cong\scO_Y$ by the vanishing theorem.
Because $\scO_X$ is a direct summand of 
the tilting generator $\scE$ constructed in 
Theorem \ref{thm:rel.dim2} we obtain 
$$
\Hom_X^i(\scE,\scO_X)=\Hom_X^i(\scO_X,\scE)=0
$$ 
for $i\ne0$. We can apply the above proposition.

\end{proof}


\noindent
Yukinobu Toda

Institute for the Physics and 
Mathematics of the Universe (IPMU), 
University of Tokyo, 
Kashiwano-ha 5-1-5, Kashiwa City, Chiba, 277-8582, Japan

{\em e-mail address}\ : \  toda@ms.u-tokyo.ac.jp

\hspace{5mm}

\noindent
Hokuto Uehara

Department of Mathematics
and Information Sciences,
Tokyo Metropolitan University,
1-1 Minamiohsawa,
Hachioji-shi,
Tokyo,
192-0397,
Japan 

{\em e-mail address}\ : \  hokuto@tmu.ac.jp

\end{document}